\documentclass[a4paper,11pt]{article}

\usepackage{fullpage} 
\usepackage{parskip} 
\usepackage{tikz} 
\usepackage{amsmath}
\usepackage{hyperref}
\usepackage{graphicx} 
\usepackage{subcaption}

\usepackage{enumerate}  
\usepackage{amsmath}    
\usepackage{amsthm}     
\usepackage{mathrsfs}   
\usepackage{amsfonts}
\usepackage{amssymb}
\usepackage{mathtools} 
\usepackage{verbatim}	
\usepackage{color}	
\usepackage{hyperref} 
\usepackage[T1]{fontenc}
\usepackage[pagewise]{lineno}
\theoremstyle{plain}                
\newtheorem{thm}{Theorem}
\newtheorem{lem}[thm]{Lemma}        
\newtheorem{prop}[thm]{Proposition}
\newtheorem{fact}[thm]{Fact}
\newtheorem{cor}[thm]{Corollary}

\theoremstyle{definition}
\newtheorem{defn}[thm]{Definition}

\theoremstyle{remark}
\newtheorem{rem}[thm]{Remark}

\DeclareMathOperator{\Arg}{Arg}   

\DeclareMathOperator{\sgn}{sgn}

\providecommand{\abs}[1]{\lvert#1\rvert}

\newcommand{\overbar}[1]{\mkern 1.5mu\overline{\mkern-1.5mu#1\mkern-1.5mu}\mkern 1.5mu} 
\newcommand{\vertiii}[1]{{\left\vert\kern-0.25ex\left\vert\kern-0.25ex\left\vert #1 
										 \right\vert\kern-0.25ex\right\vert\kern-0.25ex\right\vert}} 

\def\re{\mathop{\rm Re}\nolimits}
\def\im{\mathop{\rm Im}\nolimits}
\def\exp{\mathop{\rm e}\nolimits}

\def\CC{\mathbb C}
\def\RR{\mathbb R}
\def\NN{\mathbb N}
\def\DD{\mathbb D}

\def\calL{\mathcal{L}}

\newlength\Colsep
\title{Conditions for asymptotic stability of first order scalar differential-difference equation with complex coefficients}
\author{Rafał Kapica\footnote{Faculty of Applied Mathematics, AGH University of Science and Technology, al. Mickiewicza 30, 30-059 Kraków}, Radosław Zawiski\footnote{Department of Automatic Control and Robotics, Silesian University of Technology, ul. Akademicka 16, 44-100 Gliwice}}
\date{ }

\begin{document}

\maketitle
\abstract
We investigate a scalar characteristic exponential polynomial with complex coefficients associated with a first order scalar differential-difference equation. Our analysis provides necessary and sufficient conditions for allocation of the roots in the complex open left half-plane what guarantees asymptotic stability of the differential-difference equation. The conditions are expressed explicitly in terms of complex coefficients of the characteristic exponential polynomial, what makes them easy to use in applications. We show examples including those for retarded PDEs in an abstract formulation.


\noindent {\bf 2020 Subject Classification:} 30C15, 34K06, 34K20, 34K41

\section{Introduction}
In this article we study the asymptotic stability of a scalar linear differential-difference equation (DDE) 
\begin{equation}\label{eq:differential-difference_eq_scalars}
x'(t)=\lambda x(t) +\gamma x(t-\tau),\quad t\geq0,
\end{equation}
where $\lambda,\gamma\in\CC$ and $0<\tau$, through the analysis of the corresponding characteristic equation
\begin{equation}\label{eq:characteristic_equation_general}
s-\lambda-\gamma\exp^{-s\tau}=0.
\end{equation}
This problem is frequently related to stability analysis of 
\begin{equation}\label{eq:differential-difference_equation_general}
x'(t)=f\left(x(t),x(t-\tau)\right),\quad t\geq0,
\end{equation}
where $f:\RR^n\times\RR^n\to\RR^n$ is a smooth (nonlinear) function about $x_0\in\RR^n$, via its linearization given by
\begin{equation}\label{eq:differential-difference_equation_general_linearized}
x'(t)=Ax(t)+Bx(t-\tau),
\end{equation}
where $A=J_1f(x_0,x_0)$ and $B=J_2f(x_0,x_0)$ are partial Jacobian matrices of $f$ at $(x_0,x_0)$. To be more precise we use the following
\begin{defn}\label{def:asymptotic_stability}
An equilibrium solution $x^*(t)\equiv x_0\in\RR^n$ of \eqref{eq:differential-difference_equation_general} is exponentially stable if there exist $M,\omega,\delta>0$ such that $\|x(t)-x_0\|\leq M\exp^{-\omega t}$ $(t\geq0)$ holds for every solution $x$ of \eqref{eq:differential-difference_equation_general} satisfying an inital condition $\|x(t)-x_0\|<\delta$ $(t\in[-\tau,0])$ with the Euclidean norm $\|\cdot\|$.
\end{defn}
%
%
By the \textit{principle of linearized stability} \cite{Diekmann_et_al_1995} for the case at hand we have \cite{Nishiguchi_2016}
\begin{fact}
Let the linearization of \eqref{eq:differential-difference_equation_general} about an equilibrium solution $x^*(t)\equiv x_0$ be expressed by \eqref{eq:differential-difference_equation_general_linearized} and let the corresponding characteristic equation be given by
\begin{equation}\label{eq:characteristic_equation_general_linearized}
\det\left[sI-A-B\exp^{-s\tau}\right]=0. 
\end{equation}
Then the following statements hold:
\begin{enumerate}[(i)]
\item $x^*$ is exponentially stable if $\re s<0$ for all characteristic roots $s$ of \eqref{eq:characteristic_equation_general_linearized},
\item $x^*$ is unstable if $\re s>0$ for all characteristic roots $s$ of \eqref{eq:characteristic_equation_general_linearized}.
\end{enumerate}
\end{fact}
In a finite-dimensional setting and under appropriate, though restrictive conditions, for example if $A$ and $B$ commute - see \cite{Motzkin_Taussky_1952} or  \cite{Stepan_1989} - the problem of asymptotic stability of \eqref{eq:differential-difference_equation_general} is equivalent to finding conditions on the coefficients of \eqref{eq:differential-difference_eq_scalars} which will guarantee that every root $s$ of \eqref{eq:characteristic_equation_general} is such that $\re(s)<0$. In this setting coefficients $\lambda$ and $\gamma$ in \eqref{eq:characteristic_equation_general} are eigenvalues of $A$ and $B$, respectively. 

Equation \eqref{eq:differential-difference_eq_scalars} is encountered also in an infinite-dimensional setting. Consider \eqref{eq:differential-difference_equation_general_linearized} on a Hilbert space $X$ where $A$ is a diagonal generator of a strongly continuous semigroup and $B$ is linear and bounded diagonal operator on $X$ (see Example 4 below, \cite{Kapica_Partington_Zawiski_2022} or \cite{Partington_Zawiski_2019}). Then \eqref{eq:differential-difference_eq_scalars} describes dynamics of \eqref{eq:differential-difference_equation_general_linearized} about a single coordinate that corresponds to a given eigenvalue $\lambda$ of $A$. This setting is not a mere generalisation of the previous one. In finite dimensions it is something rather special that linearization of \eqref{eq:differential-difference_equation_general} produces commuting $A$ and $B$. In infinite dimensions surprisingly many dynamical system are represented by diagonal generators \cite{Tucsnak_Weiss}.

Our motivation to investigate \eqref{eq:differential-difference_eq_scalars} with complex coefficients thus stems from the fact that in both cases above these coefficients represent eigenvalues of respective operators. We also note that whenever mentioning stability we refer to a situation when all the roots of \eqref{eq:characteristic_equation_general} have negative real parts.

The literature contains two intertwined approaches to stability problem - one based on analysis of some form of \eqref{eq:differential-difference_equation_general_linearized} in time domain and one based on analysis of \eqref{eq:characteristic_equation_general_linearized}. In the latter approach the case with $\lambda,\,\gamma\in\RR$ is well understood - see \cite{Hayes_1950}, where the author obtained necessary and sufficient conditions for stability of $s-a-c\exp^{-s}=0$ with $a,c\in\RR$. In a more general form $A(s)+B(s)\exp^{-s\tau}=0$, with $A(s)$ and $B(s)$ real polynomials see \cite{Partington_2004} and references therein. For a thorough exposition of other methods of analysis of the real $\lambda$ and $\gamma$ case see \cite{Michiels_Niculescu_2014} and references therein. 

The case $\lambda,\gamma\in\CC$ is less analysed. In particular, some sufficiency results can be found in \cite{Barwell_1975}, where the author presents a numerical analysis of \eqref{eq:differential-difference_eq_scalars} and stipulates asymptotic stability for every $\tau>0$ if $-\re\lambda>|\gamma|$. The authors of \cite{Cahlon_Schmidt_2002} provide, based on algorithmic criteria, some sufficiency result for specific values of complex $\lambda$ and $\gamma$, proving also the result in \cite{Barwell_1975} for some cases. Authors of \cite{Wei_Zhang_2004} built on \cite{Cahlon_Schmidt_2002} and provide additional sufficient conditions for stability. In \cite{Matsunaga_2008} the author uses a continuous dependence of the roots of \eqref{eq:differential-difference_eq_scalars} on $\tau$ and manages to obtain stability conditions for some values of $\lambda\in\RR$ and $\gamma\in\CC$. In \cite{Nishiguchi_2016} the author provides necessary and sufficient conditions for the zeros of \eqref{eq:characteristic_equation_general} to be in the left complex half-plane. The argument there is based on analysis of the Lambert W function, what complicates applications of obtained conditions. In particular, the condition from \cite[Theorem 1.2]{Nishiguchi_2016} uses a nested trigonometric functions of $\im\lambda$ and $\Arg\gamma$, what makes it difficult to visualise the region in coefficients-plane that ensures stability.  

To the best of authors' knowledge \cite{Noonburg_1969} is the first work providing necessary and sufficient conditions for stability of \eqref{eq:characteristic_equation_general} for $\lambda,\,\gamma\in\CC$ and $\tau=1$. Results of  \cite{Noonburg_1969} are, however, based on specific analysis of roots of \eqref{eq:characteristic_equation_general} which is uneasy to trace for different values of $\tau$, even after change of parameters $a\mapsto a\tau$ and $\eta\mapsto\eta\tau$ (see \eqref{eq:complex_exponential_polynomial_zeros} below). This may explain why, although it precedes many of the works mentioned above, \cite{Noonburg_1969}  did not receive much recognition. 

Our approach here combines analysis of roots placement depending on $\tau$, as shown in \cite[Proposition 6.2.3]{Partington_2004}, with arguments of algebraic nature in the complex plane. This allowed us to obtain necessary and sufficient conditions for stability of \eqref{eq:characteristic_equation_general} based explicitly on a relation between $\lambda,\,\gamma\in\CC$ and $\tau>0$. The conditions do not require to calculate any specific roots of a transcendental equation and allow to visualise how the "stability" region changes with parameters in the coefficients-plane. Thus we not only provide a different formulation of stability conditions but our results are based on a new, different proof.

\section{Preliminaries}
The following observation, which can be found e.g. in \cite{Cahlon_Schmidt_2002} or \cite{Noonburg_1969}, is crucial to simplify the problem of analysis of \eqref{eq:differential-difference_eq_scalars}.
\begin{lem}\label{lem:complex_exponential_poly_equivalent_conditions}
Let $a,b,c,d,\tau\in\RR$ and let $\{s_0\}$ be the set of roots of 
\begin{equation}\label{eq:complex_expo_poly_lem_1}
s-(a+ib)-(c+id)\exp^{-s\tau}=0
\end{equation}
and $\{z_0\}$ be the set of roots of
\begin{equation}\label{eq:complex_expo_poly_lem_2}
z-a-\exp^{-ib\tau}(c+id)\exp^{-z\tau}=0.
\end{equation}
Then $\re(s_0)<0$ for all $s_0$ if and only if $\re(z_0)<0$ for all $z_0$.
\end{lem}
\begin{proof}
Let $z_0$ be a root of \eqref{eq:complex_expo_poly_lem_2}. Then $s_0=z_0+ib$ is a root of \eqref{eq:complex_expo_poly_lem_1}. Conversly, let $s_0$ be a root of \eqref{eq:complex_expo_poly_lem_1}. Then $z_0=s_0-ib$ is a root of \eqref{eq:complex_expo_poly_lem_2}. As $\re(s_0)=\re(z_0)$ the result follows.
\end{proof}
\begin{rem}
It is worth to mention that a similar yet different simplification is possible. In \cite{Breda_2012} the author based his approach on a version of Lemma~\ref{lem:complex_exponential_poly_equivalent_conditions} where the non-delayed coefficient in \eqref{eq:complex_expo_poly_lem_2} is complex and the one corresponding to the delay is real. For our approach, however, the current version of Lemma~\ref{lem:complex_exponential_poly_equivalent_conditions} is more convenient.
\end{rem}
We will also use the following result concerning parameter $\beta\in\RR$ and real functions
\begin{equation}\label{eq:rational_trig_aux_functions}
L,\,R:[0,\infty)\to\RR, \quad L(r):=\frac{r}{r^2+1},\quad R(r):=\arctan(r)+\beta.
\end{equation}
\begin{lem}\label{lem:rational_trig_eqaution}
Let $\beta\in\RR$ and put
	$$
	A=\{r\in [0,\infty):L(r)\leq R(r)\},
	$$
	where real functions $L$ and $R$ are given by \eqref{eq:rational_trig_aux_functions}. Then: 
	\begin{enumerate}[(i)]
		\item $A=[0,\infty)$ if and only if $\beta\geq 0$,
		\item $A=[r_0,\infty)$ with $r_0>0$  if and only if $\beta\in\left(-\frac{\pi}{2},0\right)$, wherein the correspondence $(0,\infty)\ni r_0\longleftrightarrow \beta\in \left(-\frac{\pi}{2},0\right)$ is one-to-one,
		\item set $A$ is empty if and only if $\beta\leq -\frac{\pi}{2}$.
	\end{enumerate}
\end{lem}
\begin{proof}
	Let us consider function $\varphi:[0,\infty)\to\RR$ given by $\varphi=R-L$. Clearly, we have 
		\[
	\varphi'(r)=\frac{2r^2}{(r^2+1)^2}>0,
	\quad \varphi(0)=\beta, 
	\quad \varphi(r)\longrightarrow\beta+\frac{\pi}{2}
	\quad \rm{as}\quad r\rightarrow\infty. 
	\]	
	In particular $\varphi$ is strictly increasing and $\varphi([0,\infty))=[\beta,\beta+\frac{\pi}{2})$.
	
	If $\beta\geq 0$, then $\varphi(r)\geq 0$ for $r\in [0,\infty)$, i.e.
	$A=[0,\infty)$. On the other hand if $L(0)\leq R(0)$, then $\beta\geq 0$. This gives assertion (i). 
	
	Suppose $\beta\in\left(-\frac{\pi}{2},0\right)$. Hence $\varphi(0)<0$ and $\varphi(r)>0$ for large enough $r>0$. Then there exists a unique $r_0>0$ such that $\varphi(r_0)=0$. This shows that $A=[r_0,\infty)$.
	If now  $A=[r_0,\infty)$ for some $r_0>0$, then $\beta<0$ by (i).  To finish the proof it is enough to notice that for  $\beta\leq\frac{\pi}{2}$ we have $\varphi(r)<0$ for $r\in [0,\infty)$, i.e. $A=\emptyset$.
	\end{proof}
\begin{cor}\label{cor:rational_trig_eqaution_number_of_solutions}
Equation $L(r)=R(r)$ has exactly one solution if and only if $\beta\in(-\frac{\pi}{2},0]$.\qed
\end{cor}
We also make use of the following half-planes
\begin{align*}
\CC_+:=&\{s\in\CC:\re(s)>0\},\quad \CC_-:=\{s\in\CC:\re(s)<0\},\\
\Pi_+:=&\{s\in\CC:\im(s)>0\},\quad \Pi_-:=\{s\in\CC:\im(s)<0\}.
\end{align*}
\section{Main results}
By Lemma~\ref{lem:complex_exponential_poly_equivalent_conditions} we restrict our attention to \eqref{eq:complex_expo_poly_lem_2}. Taking $\eta=u+iv=\exp^{-ib\tau}(c+id)$ the conditions for stability of \eqref{eq:characteristic_equation_general} are given on an $(u,iv)$-complex plane in terms of regions that depend on $a$ and $\tau$.
\begin{rem}
We take the principal argument of $\lambda$ to be $\Arg\lambda\in(-\pi,\pi]$.
\end{rem}
Let $\DD_r\subset\CC$ be an open disc centred at $0$ with radius $r>0$. We shall require the following subset of the complex plane, depending on $\tau>0$ and $a\in(-\infty,\frac{1}{\tau}]$, namely: \index{Greek@\textbf{Greek symbols}!k@$\Lambda_{\tau,a}$}
\begin{itemize}
\item for $a<0$:
\begin{align}\label{eq:lambda_tau_a_neg}
\begin{split}
\Lambda_{\tau,a}:=&\bigg\{\eta\in\CC\setminus\DD_{|a|}:\re{\eta}+a<0,\,|\eta|<|\eta_\pi|,\\
&\abs{\Arg\eta}>\tau\sqrt{\abs{\eta}^2-a^2}+\arctan\Big(-\frac{1}{a}\sqrt{\abs{\eta}^2-a^2}\Big)\bigg\}\cup\DD_{|a|},
\end{split}
\end{align}
where $\eta_\pi$ is such that
\[
\sqrt{\abs{\eta_\pi}^2-a^2}\tau+\arctan\bigg(-\frac{1}{a}\sqrt{\abs{\eta_\pi}^2-a^2}\bigg)=\pi;
\]
\item for $a=0$:
\begin{align}\label{eq:lambda_tau_a_zero}
\Lambda_{\tau,a}:=&\bigg\{\eta\in\CC\setminus\{0\}:\re\eta<0,\, |\eta|<\frac{\pi}{2\tau},\, \abs{\Arg\eta}>\tau\abs{\eta}+\frac{\pi}{2}\bigg\};
\end{align}
\item for $0<a\leq\frac{1}{\tau}$
\begin{align}\label{eq:lambda_tau_a_pos}
\begin{split}
\Lambda_{\tau,a}:=&\bigg\{\eta\in\CC:\re{\eta}+a<0,\,|\eta|<|\eta_\pi|,\\
&\abs{\Arg\eta}>\tau\sqrt{\abs{\eta}^2-a^2}+\arctan\Big(-\frac{1}{a}\sqrt{\abs{\eta}^2-a^2}\Big)+\pi\bigg\},
\end{split}
\end{align}
where $\eta_\pi$ is such that $\abs{\eta_\pi}>a$ and
\[
\sqrt{\abs{\eta_\pi}^2-a^2}\tau+\arctan\bigg(-\frac{1}{a}\sqrt{\abs{\eta_\pi}^2-a^2}\bigg)=0.
\] 
\end{itemize}
Figures \ref{fig:Lambda_tau_a_neg_different_tau}, \ref{fig:Lambda_tau_a_zero_different_tau} and \ref{fig:Lambda_tau_a_pos_different_tau} show $\Lambda_{\tau,a}$ for fixed values of $a$ and varying $\tau$, while Figure \ref{fig:Lambda_tau_a_for_different_a} shows $\Lambda_{\tau,a}$ for fixed $\tau$ and varying $a$. 

\begin{figure}[!h]
    \begin{minipage}[l]{0.48\linewidth}
        \centering
        \includegraphics[scale=0.5]{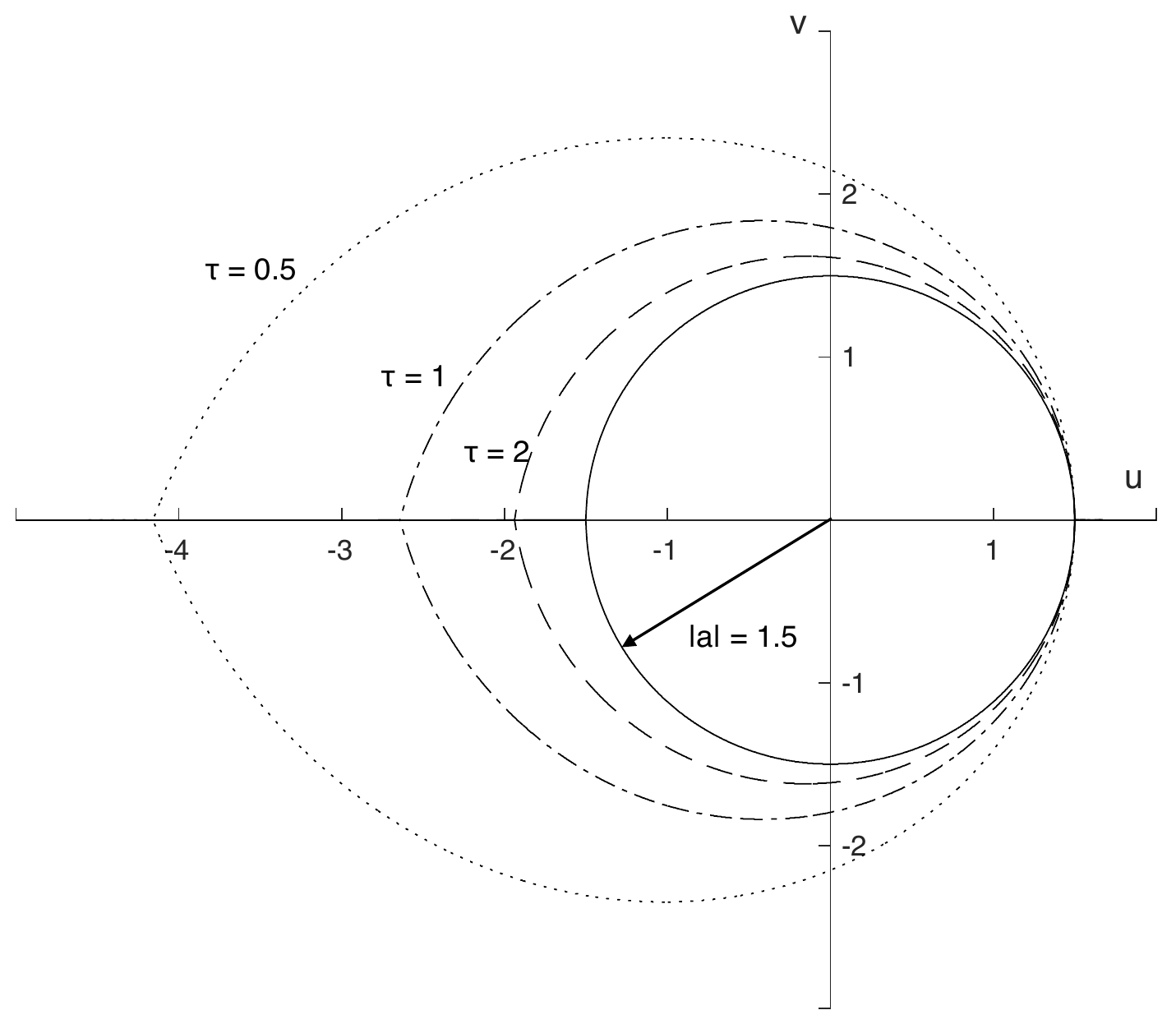} 
        \caption{Outer boundaries of the $\Lambda_{\tau,a}$, defined in \eqref{eq:lambda_tau_a_neg} with $\eta=u+iv$, for $a=-1.5$ and different values of $\tau$: dotted for $\tau=0.5$, dash-dotted for $\tau=1$, dashed for $\tau=2$. The solid line shows a circle with radius $|a|=1.5$.}
        \label{fig:Lambda_tau_a_neg_different_tau}
    \end{minipage}\hfill
    \begin{minipage}[c]{0.48\linewidth}
        \centering
        \includegraphics[scale=0.5]{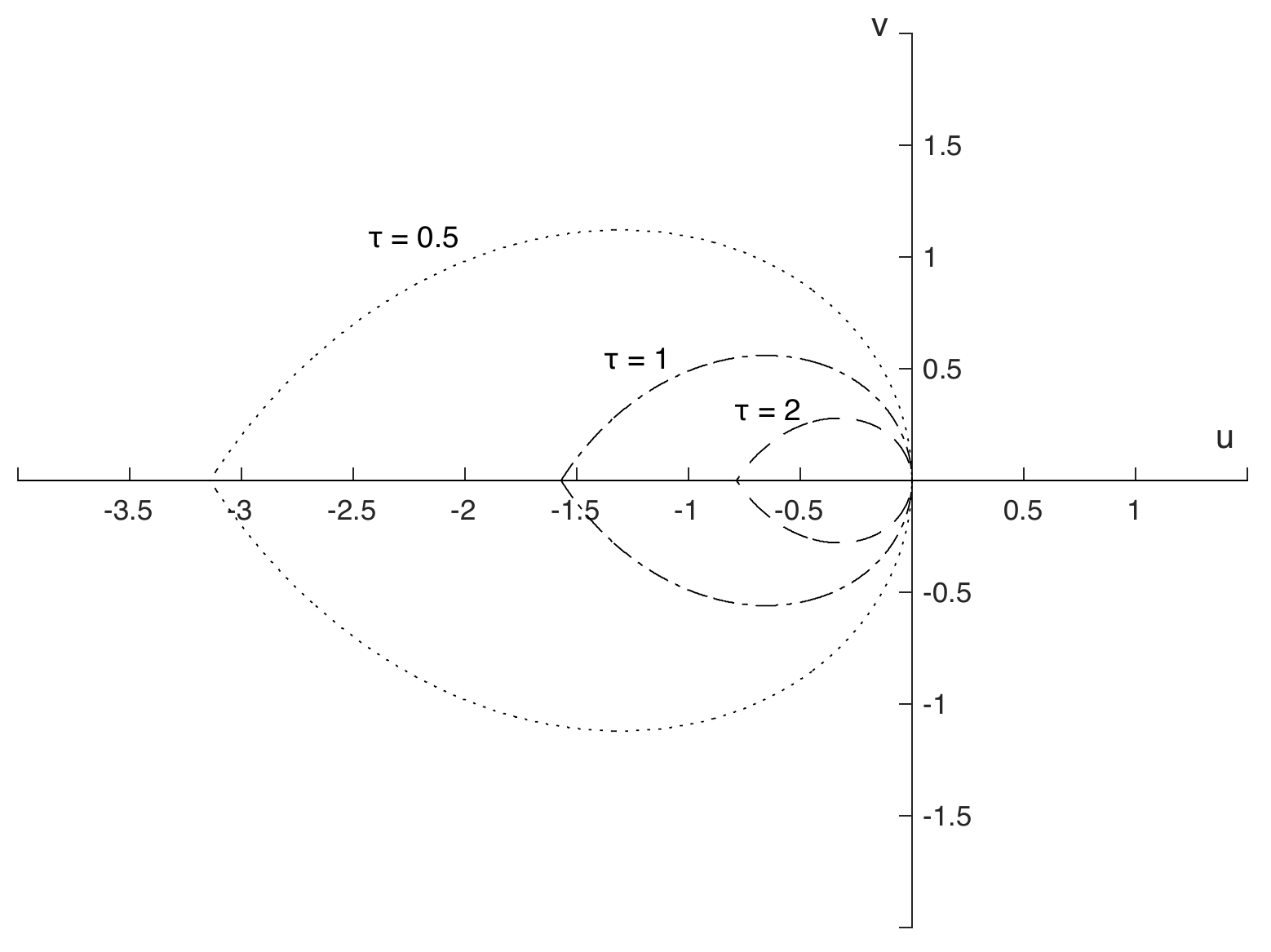} 
        \caption{Outer boundaries of the $\Lambda_{\tau,a}$, defined in \eqref{eq:lambda_tau_a_zero} with $\eta=u+iv$, for $a=0$ and different values of $\tau$: dotted for $\tau=0.5$, dash-dotted for $\tau=1$, dashed for $\tau=2$.}
\label{fig:Lambda_tau_a_zero_different_tau}
    \end{minipage}
\end{figure}

\begin{figure}[!h]
    \begin{minipage}[l]{0.48\linewidth}
        \centering
        \includegraphics[scale=0.5]{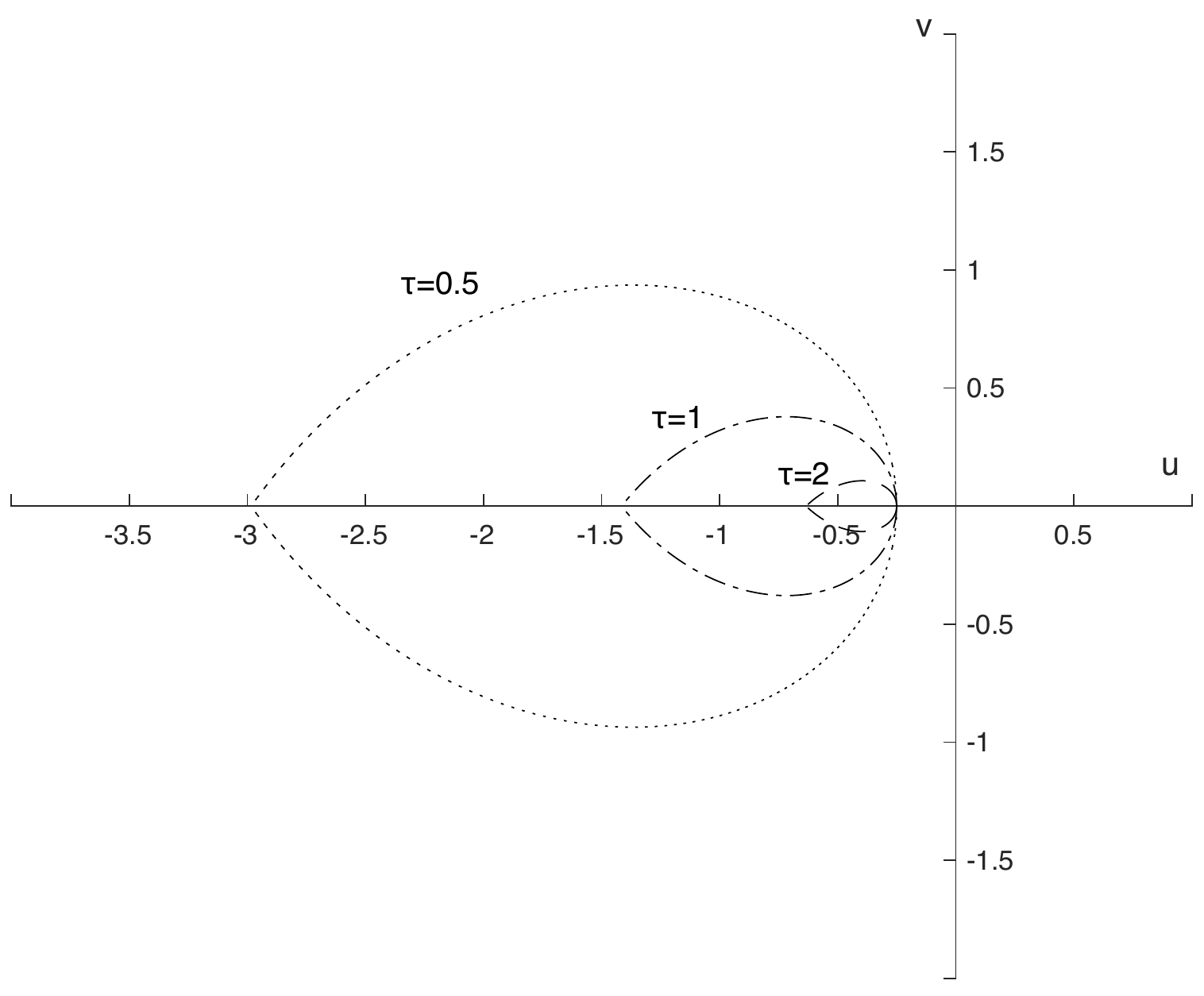} 
        \caption{Outer boundaries of the $\Lambda_{\tau,a}$, defined in \eqref{eq:lambda_tau_a_pos} with $\eta=u+iv$, for $a=0.25$ and different values of $\tau$: dotted for $\tau=0.5$, dash-dotted for $\tau=1$, dashed for $\tau=2$.}
        \label{fig:Lambda_tau_a_pos_different_tau}
    \end{minipage}\hfill
    \begin{minipage}[c]{0.48\linewidth}
        \centering
        \includegraphics[scale=0.5]{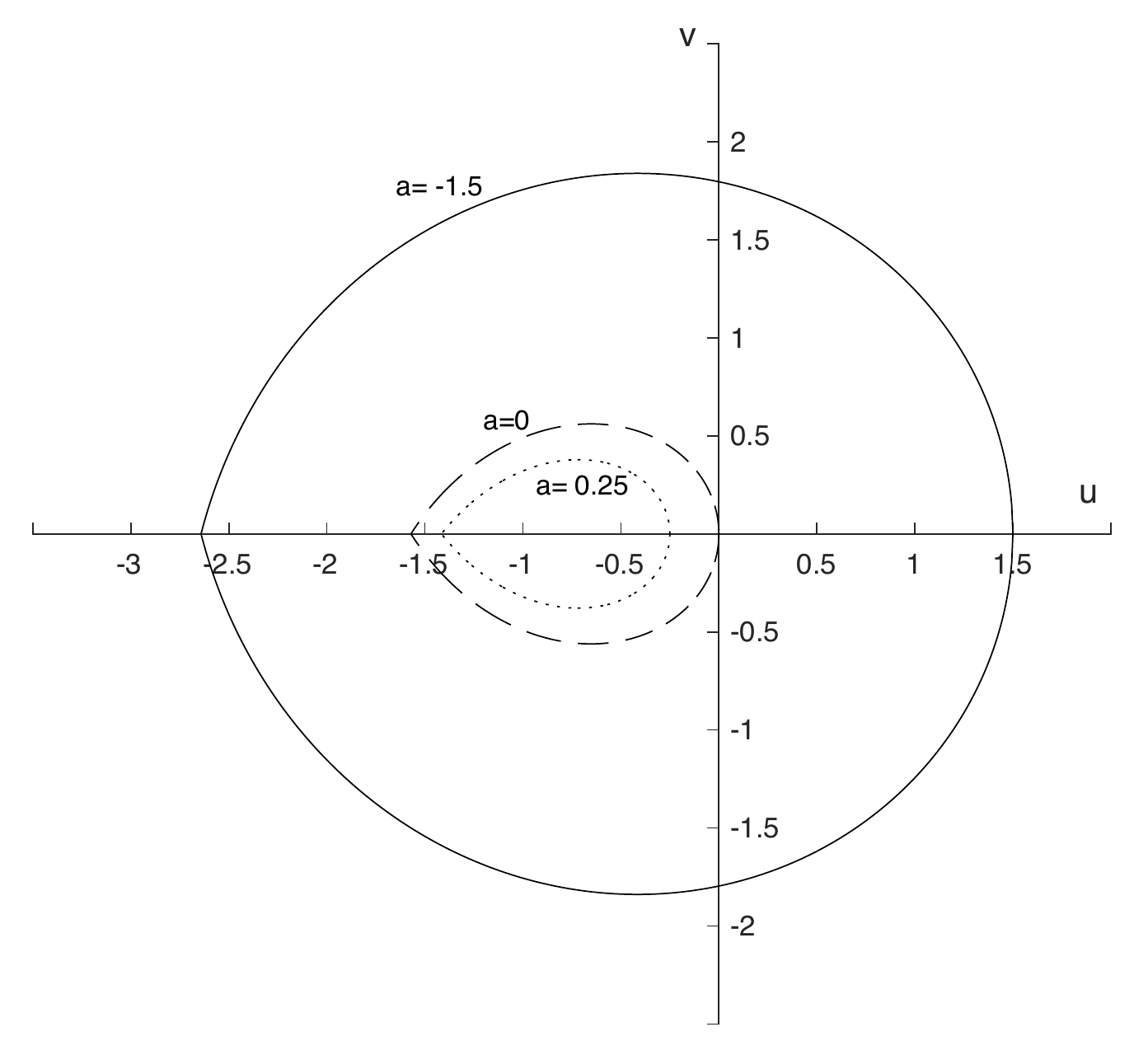} 
        \caption{Outer boundaries of $\Lambda_{\tau,a}$, defined in \eqref{eq:lambda_tau_a_neg}--\eqref{eq:lambda_tau_a_pos} with $\eta=u+iv$, for $\tau=1$ and different values of $a$: solid for $a=-1.5$, dashed for $a=0$ and dotted for $a=0.25$.}
\label{fig:Lambda_tau_a_for_different_a}
    \end{minipage}
\end{figure}

The zeros of \eqref{eq:characteristic_equation_general} are in the left half plane $\CC_-$, according to Lemma~\ref{lem:complex_exponential_poly_equivalent_conditions}, if and only if the roots of  \eqref{eq:complex_expo_poly_lem_2} belong to $\CC_-$. Thus for $\lambda=a+ib$, $\gamma=c+id$ and $\eta:=\exp^{-ib\tau}\gamma$ we have the following

\begin{thm}\label{thm:stability_of_complex_polynomial}
Let $\tau>0$, $a\in\RR$ and $\eta\in\CC$. Then every solution of the equation
\begin{equation}\label{eq:complex_exponential_polynomial_zeros}
s-a-\eta\exp^{-s\tau}=0
\end{equation}
belongs to $\CC\setminus\CC_+$ if and only if $a\leq\frac{1}{\tau}$ and  $\eta$ belongs to the closure of $\Lambda_{\tau,a}$ given by \eqref{eq:lambda_tau_a_neg}, \eqref{eq:lambda_tau_a_zero} or \eqref{eq:lambda_tau_a_pos}.
\end{thm}
\begin{proof}
\begin{itemize} 
\item[1.] Denote the closure of $\Lambda_{\tau,a}$ by $\overbar{\Lambda_{\tau,a}}$. For any $\tau>0$ and $a\leq0$ there is $0\in\overbar{\Lambda_{\tau,a}}$ and taking $\eta=0$ the statement of the proposition obviously holds true, while for $0<a\leq\frac{1}{\tau}$ we have $0\not\in\overbar{\Lambda_{\tau,a}}$. Thus for the remainder of the proof assume that $\eta\neq0$.

\item[2.] It is know that for $\tau>0$ equation \eqref{eq:complex_exponential_polynomial_zeros} has infinitely many solutions. By the Rouch{\'e}'s theorem (see e.g. \cite[Prop.1.14]{Michiels_Niculescu_2014}) solutions of \eqref{eq:complex_exponential_polynomial_zeros} vary continuously with $\tau$, except at $\tau=0$ where only one remains. Let $\eta=u+iv$ and let $s=x+iy$. In the limit as $\tau\rightarrow 0$  in \eqref{eq:complex_exponential_polynomial_zeros} we obtain
\[x=a+u,\quad y=v,
\]
and the solutions start in $\CC\setminus\CC_+$ i.e. with $x\leq0$ if and only if $a+u\leq0$. 

Let us establish when at least one of the solutions crosses the imaginary axis for the first time as $\tau$ increases from zero upwards. At the crossing of the imaginary axis there is $s=i\omega$ for some $\omega\in\RR$. In view of \eqref{eq:complex_exponential_polynomial_zeros} we can treat $s$ as an implicit function of $\tau$ and check the direction in which zeros of it cross the imaginary axis  by analysing the $\sgn\re\frac{ds}{d\tau}$ if $s=i\omega$. By calculating the implicit function derivative we have
	\[
	\frac{ds}{d\tau}=-\frac{s^2-as}{1-a\tau+s\tau}.
	\]
  As $\sgn\re z=\sgn\re z^{-1}$ we have if $s=i\omega$ that
  	\[
  	\sgn\re\frac{ds}{d\tau}=\sgn\frac{1}{\omega^2+a^2}>0
	\]
	and the zeros cross from the left to the right half-plane. As the sign of the above does not depend on $\tau$, the direction of the crossing remains the same for every value of $\tau$. Thus with $\eta=u+iv$ a necessary condition for the solutions of \eqref{eq:complex_exponential_polynomial_zeros} to be in $\CC\setminus\CC_+$ is
\begin{equation}\label{eq:stability_of_complex_poly_necessary_condition}
a+u\leq0.
\end{equation}
\item[3.] Consider again \eqref{eq:complex_exponential_polynomial_zeros} with fixed $\tau>0$ and $a\in\RR$ and take such $\eta\in\CC$ that \eqref{eq:stability_of_complex_poly_necessary_condition} holds. Let us focus on the crossing point i.e. let $s=i\omega$ for some $\omega\in\RR$. Taking the complex conjugate of \eqref{eq:complex_exponential_polynomial_zeros} at the crossing we obtain
  \begin{equation}\label{eq:complex_conjugate_exponential_polynomial_zeros}
  -i\omega-a-\bar{\eta}\exp^{i\omega\tau}=0.
  \end{equation}
 
Using now \eqref{eq:complex_exponential_polynomial_zeros} for $s=i\omega$ and \eqref{eq:complex_conjugate_exponential_polynomial_zeros} to eliminate the exponential part we have $  \omega^2=\abs{\eta}^2-a^2$. From here we see that for a given $a\in\RR$ and every $\eta=u+iv$ satisfying both, \eqref{eq:stability_of_complex_poly_necessary_condition} and $\abs{\eta}<\abs{a}$,
%
%
the crossing does not exist, regardless of $\tau$, and all the solutions of \eqref{eq:complex_exponential_polynomial_zeros} are in $\CC\setminus\CC_+$. 
\item[4.] Let us focus on the case when the first crossing happens. To that end consider \eqref{eq:complex_exponential_polynomial_zeros}, fix $a\in\RR$ and take $\eta$ such that \eqref{eq:stability_of_complex_poly_necessary_condition} and
  \begin{equation}\label{eq:eta_condition_2}
  \abs{\eta}\geq\abs{a}
  \end{equation}

hold. By point 2 as $\tau$ increases the roots of \eqref{eq:complex_exponential_polynomial_zeros} move continuously to the right. Denote by $\tau_0$ the smallest $\tau$ for which the crossing happens. By assumptions we know that such $\tau_0$ exists. By point 3 the crossing takes place at $s=\pm i\sqrt{\abs{\eta}^2-a^2}$. Putting $s=i\sqrt{\abs{\eta}^2-a^2}$  into \eqref{eq:complex_exponential_polynomial_zeros}  with $\tau=\tau_0$ gives 
 \begin{equation}\label{eq:stability_of_complex_poly_step_1+}
 \begin{split}
  \eta=&-a\exp^{i\sqrt{\abs{\eta}^2-a^2}\tau_0}+\sqrt{\abs{\eta}^2-a^2}\exp^{i(\frac{\pi}{2}+\sqrt{\abs{\eta}^2-a^2}\tau_0)}\\
   =&\exp^{i\sqrt{\abs{\eta}^2-a^2}\tau_0}\left(-a+i\sqrt{\abs{\eta}^2-a^2}\right),
  \end{split}
  \end{equation}
Putting $s=-i\sqrt{\abs{\eta}^2-a^2}$ into \eqref{eq:complex_exponential_polynomial_zeros}  gives an equation corresponding to \eqref{eq:stability_of_complex_poly_step_1+}, namely
\begin{equation}\label{eq:stability_of_complex_poly_step_1-}
 \eta=\exp^{-i\sqrt{\abs{\eta}^2-a^2}\tau_0}\left(-a-i\sqrt{\abs{\eta}^2-a^2}\right).
\end{equation}
Equations  \eqref{eq:stability_of_complex_poly_step_1+} and \eqref{eq:stability_of_complex_poly_step_1-} show the relation between all coefficients (or parameters) of \eqref{eq:complex_exponential_polynomial_zeros} in the boundary case of transition between asymptotic stability and instability. Thus we focus on  the triple $(\tau_0,a,\eta)$ and how changes within it influence stability of 
\begin{equation}\label{eq:complex_exponential_polynomial_zeros_tau0}
s-a-\eta\exp^{-s\tau_0}=0.
\end{equation}
By point 2, for $a$ and $\eta$ as in \eqref{eq:stability_of_complex_poly_step_1+} and with every $\tau>\tau_0$ equation \eqref{eq:complex_exponential_polynomial_zeros_tau0} is unstable, while for $\tau<\tau_0$ it is stable. And so we turn our attention to $\eta$.  

\item[5.] Let $\eta$ be a solution of \eqref{eq:stability_of_complex_poly_step_1+}. Then $\overbar{\eta}$ is a solution of \eqref{eq:stability_of_complex_poly_step_1-} and these solutions are obviously symmetric about the real axis. It will be more convenient to use different notation that the one in \eqref{eq:stability_of_complex_poly_step_1+} or \eqref{eq:stability_of_complex_poly_step_1-}. Define $\gamma_+:[|a|,\infty)\to\CC$ and $\gamma_-:[|a|,\infty)\to\CC$ as the right side of \eqref{eq:stability_of_complex_poly_step_1+} and \eqref{eq:stability_of_complex_poly_step_1-}, respectively,  i.e.
\begin{equation}\label{eq:stability_of_complex_poly_gamma+_path}
\gamma_+(w):=\exp^{i\sqrt{w^2-a^2}\tau_0}\left(-a+i\sqrt{w^2-a^2}\right),
\end{equation}
\begin{equation}\label{eq:stability_of_complex_poly_gamma-_path}
\gamma_-(w):=\exp^{-i\sqrt{w^2-a^2}\tau_0}\left(-a-i\sqrt{w^2-a^2}\right).
\end{equation}
Let $\Gamma_+:=\gamma_+([|a|,\infty))$ be the image of \eqref{eq:stability_of_complex_poly_gamma+_path} and $\Gamma_-:=\gamma_-([|a|,\infty))$ be the image of \eqref{eq:stability_of_complex_poly_gamma-_path}. We easily see that $\gamma_+(w)=\overbar{\gamma_-(w)}$ for every $w\geq |a|$ and so $\Gamma_+$ is symmetric to $\Gamma_-$ about the real axis. 

Up to this moment all considerations in points 2--5 were done regardless of the sign of parameter $a$. In the reminder of the proof, along with \eqref{eq:stability_of_complex_poly_necessary_condition} and \eqref{eq:eta_condition_2}, we will consider additional assumptions on $a$, namely $a<0$, $a=0$, $a\in(0,\frac{1}{\tau_0}]$ and $a>\frac{1}{\tau_0}$.

\item[6.] Assume additionally that $a<0$ and let a function describing a continuous argument increment of \eqref{eq:stability_of_complex_poly_gamma+_path} be given by $\Delta\gamma_+:[|a|,\infty)\to[0,\infty)$,
\begin{equation}\label{eq:stability_of_complex_poly_gamma+_argument}
\Delta\gamma_+(w)=\sqrt{w^2-a^2}\tau_0+\arctan\bigg(-\frac{1}{a}\sqrt{w^2-a^2}\bigg).
\end{equation}
We easily see that it is a strictly increasing, non-negative function. We also define $\Delta\gamma_-:[|a|,\infty)\to(-\infty,0]$ and $\Delta\gamma_-(w)=-\Delta\gamma_+(w)$ for every $w\in[|a|,\infty)$.

Looking at \eqref{eq:stability_of_complex_poly_gamma+_path} note that the first component has modulus $1$ and introduces counter-clockwise rotation, while the second component is always in the first quadrant, with a positive real part equal to $-a$, and its modulus is strictly increasing and tends to infinity as $w\to\infty$. Thus $\Gamma_+$ is a curve that is a counter-clockwise outward spiral that begins in $-a\in\CC$. An exemplary pair of $\Gamma_+$ and $\Gamma_-$ curves is shown in Fig. \ref{fig:Gamma_curves}. 

\begin{figure}[!b]
    \begin{minipage}[l]{0.48\linewidth}
        \centering
        \includegraphics[scale=0.63]{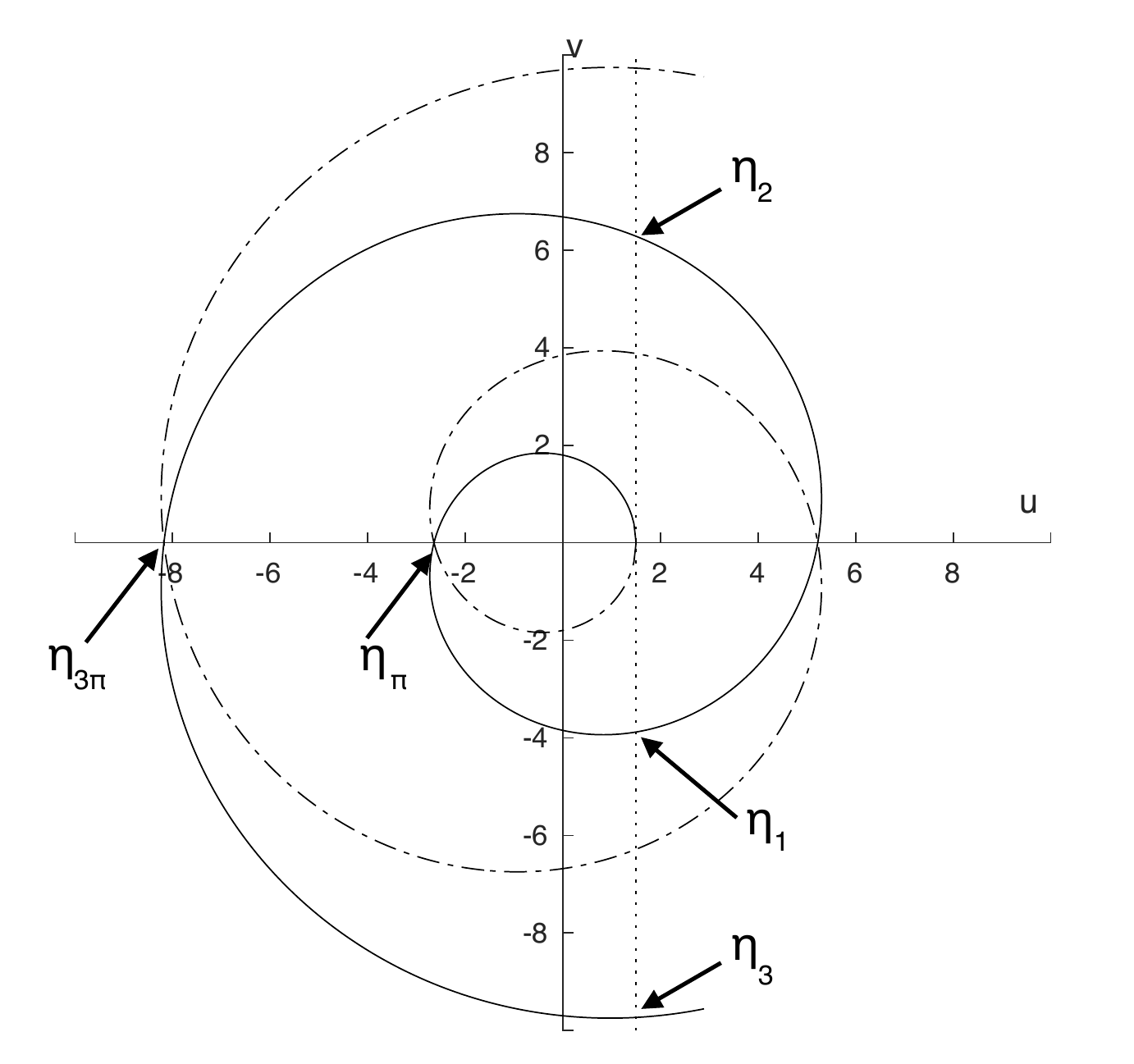} 
        \caption{Curves $\Gamma_+$ (solid line) and $\Gamma_-$ (dash-dotted line) drawn for $\tau_0=1$ and $a=-1.5$ with $|\eta|=w\in(|a|,10)$. The constraint related to $a$ and expressed by \eqref{eq:stability_of_complex_poly_necessary_condition} is marked with a dotted line. The crossings of the real negative semi-axis by $\Gamma_+$ (and $\Gamma_-$) are at $\eta_\pi$ and $\eta_{3\pi}$. The crossings of $u=-a$ by $\Gamma_+$, as $w$ increases, are at $\eta_1$, $\eta_2$ and $\eta_3$.}
      \label{fig:Gamma_curves}
    \end{minipage}\hfill
    \begin{minipage}[c]{0.48\linewidth}
        \centering
        \includegraphics[scale=0.5]{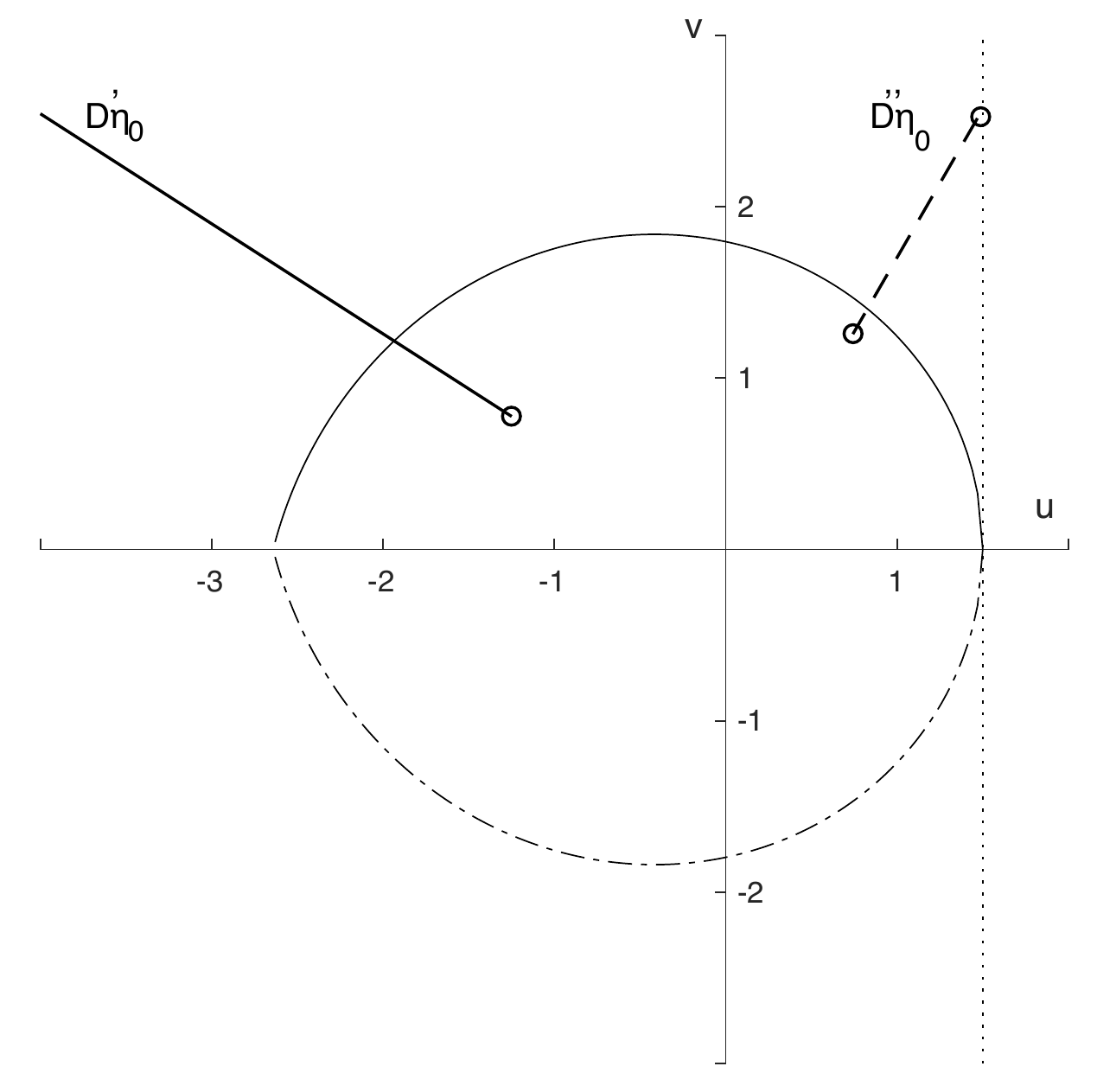} 
        \caption{Enlargement of the central part of Fig.~\ref{fig:Gamma_curves} with $\gamma_+([|a|,|\eta_\pi|])\cup \gamma_-([|a|,|\eta_\pi|])$ with two cases $D'_{\eta_0}$ (solid line) and $D''_{\eta_0}$ (dashed line). Note that $D''_{\eta_0}$ is bounded due to \eqref{eq:stability_of_complex_poly_necessary_condition}.}
\label{fig:Gamma_curves_zoom}
    \end{minipage}
\end{figure}

\item[7.] Let a set $\{\eta_{(2k-1)\pi}\}_{k\in\NN}$ be such that the argument increment along $\Gamma_+$ as $w$ changes from $|a|$ to $|\eta_{(2k-1)\pi}|$ is equal to $(2k-1)\pi$, that is
\begin{equation}\label{eq:stability_of_complex_poly_gamma+_arg_increment}
\Delta\gamma_+\left(|\eta_{(2k-1)\pi}|\right)=(2k-1)\pi.
\end{equation}
Due to constraint \eqref{eq:stability_of_complex_poly_necessary_condition} we take into account only these parts of $\Gamma_+$ (or $\Gamma_-$) that lie to the left of $u=-a$ line, as depicted in Figs. \ref{fig:Gamma_curves} and \ref{fig:Gamma_curves_zoom}. Let us now focus on the closure of the first part of $\Gamma_+$ that lies in $\Pi_+$ i.e. $\gamma_+([\,|a|,|\eta_\pi|\,])$. By \eqref{eq:stability_of_complex_poly_gamma+_argument} and \eqref{eq:stability_of_complex_poly_gamma+_arg_increment} for every $w\in[|\,a|,|\eta_\pi|\,]$ we have $\Delta\gamma_+(w)\in[0,\pi]$. For the case of the part of $\Gamma_-$ equal to $\gamma_-([\,|a|,|\eta_\pi|\,])$ the argument expression gives $\Delta\gamma_-(w)=-\Delta\gamma_+(w)$. Putting both cases together and returning to the notation of \eqref{eq:stability_of_complex_poly_step_1+} and \eqref{eq:stability_of_complex_poly_step_1-}, our equation of interest becomes
\begin{equation}\label{eq:lambda_tau_a_equality}
\abs{\Arg\eta}=\sqrt{\abs{\eta}^2-a^2}\tau_0+\arctan\bigg(-\frac{1}{a}\sqrt{\abs{\eta}^2-a^2}\bigg),\quad|\eta|\leq|\eta_\pi|,
\end{equation}
where $\eta_\pi$ is such that
\[
\Delta\gamma_+(|\eta_\pi|)=\sqrt{\abs{\eta_\pi}^2-a^2}\tau_0+\arctan\bigg(-\frac{1}{a}\sqrt{\abs{\eta_\pi}^2-a^2}\bigg)=\pi.
\]

\item[8.] The set of all $\eta\in\CC$ that satisfy \eqref{eq:lambda_tau_a_equality} is the boundary of the $\Lambda_{\tau_0,a}$ region - see Fig.~\ref{fig:Lambda_tau_a_neg_different_tau} for its shape. To show that for every $\eta$ inside this boundary the roots of \eqref{eq:complex_exponential_polynomial_zeros_tau0} are in $\CC\setminus\CC_+$ consider the following. For every $\eta$ in the half-plane $\{u+iv\in\CC:\ a+u<0\}$ simple geometric considerations show that there exists exactly one $\eta_0$ fulfilling \eqref{eq:lambda_tau_a_equality} and such that $\Arg\eta=\Arg\eta_0$. Conversely, let us fix $\eta_0$ fulfilling \eqref{eq:lambda_tau_a_equality} and consider a function $\tau=\tau(|\eta|)$ defined on a ray from the origin and passing through $\eta_0$. More precisely, define 
\[
D_{\eta_0}:=\{\eta=u+iv\in\CC:\abs{\eta}>\abs{a}\ \hbox{and}\ a+u<0\ \hbox{and}\ \Arg\eta=\Arg\eta_0\}
\]
and let $D_{\eta_0}^t:=\{t\geq0:t=|\eta|,\, \eta\in D_{\eta_0}\}$. Now reformulate the equality in \eqref{eq:lambda_tau_a_equality}  to express $\tau$ as a function $\tau:D_{\eta_0}^t\to (0,\infty)$,
\begin{equation}\label{eq:tau_as_function_of_abs_eta}
\tau(t)=\frac{\arctan\bigg(\frac{1}{a}\sqrt{t^2-a^2}\bigg)+\abs{\Arg\eta_0}}{\sqrt{t^2-a^2}}.
\end{equation}
This is a well-defined positive continuous function. Indeed, for positivity note that for $u\leq0$ there is $|\Arg\eta_0|\geq\frac{\pi}{2}$, while for $u\in(0,-a)$ consider the following trigonometric identity
\begin{equation*}
\begin{split}
&\arctan\bigg(\frac{1}{a}\sqrt{t^2-a^2}\bigg)+\abs{\Arg\eta_0}
=\arctan\bigg(\frac{1}{a}\sqrt{t^2-a^2}\bigg)+\arctan\bigg(\Big|\frac{v}{u}\Big|\bigg)\\
&=\arctan\Bigg(\frac{\frac{u}{a}\sqrt{t^2-a^2}+|v|}{u-\frac{1}{a}\sqrt{t^2-a^2}|v|}\Bigg)
\end{split}
\end{equation*}
and the estimation $\frac{u}{a}\sqrt{t^2-a^2}>-\sqrt{u^2+v^2-a^2}>-|v|$. The derivative of \eqref{eq:tau_as_function_of_abs_eta} is given by
\begin{equation}\label{eq:tau_as_function_derivative}
\frac{d\tau}{dt}(t)=\frac{t}{t^2-a^2}\bigg(\frac{a}{{t^2}}-\tau(t)\bigg).
\end{equation} 
As $a<0$ we have $\frac{d\tau}{dt}<0$ for every $t\in D_{\eta_0}^t$ and $\tau$ is a decreasing function. Thus for every $\eta\in D_{\eta_0}$ such that $\abs{\eta}\leq\abs{\eta_0}$ we have $\tau(|\eta|)\geq\tau(|\eta_0|)=\tau_0$, that is
\begin{equation}\label{eq:lambda_tau_a_condition}
\abs{\Arg\eta}\geq\sqrt{\abs{\eta}^2-a^2}\tau_0-\arctan\bigg(\frac{1}{a}\sqrt{\abs{\eta}^2-a^2}\bigg), \quad|\eta|\leq|\eta_\pi|.
\end{equation}

As the above is true for every $\eta_0$ fulfilling \eqref{eq:lambda_tau_a_equality}, condition  \eqref{eq:lambda_tau_a_condition} is true for every $\eta\in\Lambda_{\tau_0,a}\setminus\{\eta\in\CC:|\eta|\leq |a|\}$. Stated otherwise, for a given $\eta'\in\Lambda_{\tau_0,a}\setminus\{\eta\in\CC:|\eta|\leq |a|\}$ the time $\tau'$ for this $\eta'$ to be such that the first root of \eqref{eq:complex_exponential_polynomial_zeros_tau0} reaches the imaginary axis is bigger than $\tau_0$. This also gives that $\tau'>\tau_0$ implies $\Lambda_{\tau',a}\subset\Lambda_{\tau_0,a}$, as shown in Fig.~\ref{fig:Lambda_tau_a_neg_different_tau}.

\item[9.] Results of the previous point show that the only parts of $\Gamma_+$ and $\Gamma_-$ that we need to consider are the ones already discussed i.e. $\gamma_+([|a|,|\eta_\pi|])\cup \gamma_-([|a|,|\eta_\pi|])$.

Indeed, let $\eta_k$, $k=1,2,\dots$ be consecutive points where $\Gamma_+$ crosses the constraint line $u=-a$, as depicted in Figs. \ref{fig:Gamma_curves} and \ref{fig:Gamma_curves_zoom}. Then for every 
\[
\eta_+\in\gamma_+(|\eta_\pi|,|\eta_1|)\cup\gamma_+([|\eta_{2k}|,|\eta_{(2k+1)}|]),\quad k\in\NN
\]
there exists
\[
\eta_0\in\gamma_+([|a|,|\eta_\pi|])\cup \gamma_-([|a|,|\eta_\pi|])
\]
such that \[
\Arg\eta_0=\Arg\eta_+\quad\hbox{and}\quad|\eta_0|<|\eta_+|.
\] 
The result of point 8 now gives a contradiction as $\tau_0$ cannot be the smallest delay for which the first crossing happens. In fact, although \eqref{eq:stability_of_complex_poly_step_1+} still describes \eqref{eq:complex_exponential_polynomial_zeros_tau0} with a root corresponding to $\eta_+$ at the imaginary axis, at least one root of \eqref{eq:complex_exponential_polynomial_zeros_tau0} - the one corresponding to $\eta_0$ - is already in $\CC_+$. The same argument holds for $\Gamma_-$.

\item[10.] It is easy to see that  for $|\eta|=|a|$ estimation \eqref{eq:lambda_tau_a_condition} is true and thus the closed disc $\{\eta\in\CC:|\eta|\leq |a|\}\subset\overbar{\Lambda_{\tau_0,a}}$. Taking into account that for the interior of this disc the roots of \eqref{eq:complex_exponential_polynomial_zeros_tau0} are in $\CC\setminus\CC_+$ (see point 3),
we reach the necessity of the condition $\eta\in \overbar{\Lambda_{\tau_0,a}}$ for $a<0$.

\item[11.] Let now $a=0$ and let $\tau_0>0$ be as before (considerations in points 1 and 2 remain the same). The crossing takes place at $s=\pm i|\eta|$. Equation
\begin{equation}\label{eq:lambda_tau_zero_equality}
\abs{\Arg\eta}=|\eta|\tau_0+\frac{\pi}{2}
\end{equation}
comes now directly from \eqref{eq:stability_of_complex_poly_step_1+}. The analysis of points 5--9 simplifies greatly resulting in a necessity condition of the form 
\begin{equation}\label{eq:lambda_tau_zero_condition}
\abs{\Arg\eta}\geq|\eta|\tau_0+\frac{\pi}{2},\quad |\eta|<\frac{\pi}{2\tau_0}.
\end{equation}
\item[12.] Assume now $0<a$. Equations \eqref{eq:stability_of_complex_poly_gamma+_path} and \eqref{eq:stability_of_complex_poly_gamma-_path} have the same form. The difference now is that the second product term in \eqref{eq:stability_of_complex_poly_gamma+_path} is constantly in the second quadrant, with a negative real part $-a$ and imaginary part tending to $+\infty$ as $w\to\infty$. This changes e.g. the behaviour of the continuous argument increment function $\Delta\gamma_+$, as it is in general no longer strictly increasing. 

In fact for $0<a$ we have $\Delta\gamma_+:[a,\infty)\to[0,\infty)$,
\begin{equation}\label{eq:stability_of_complex_poly_gamma+_argument_a_pos}
\Delta\gamma_+(w)=\sqrt{w^2-a^2}\tau_0+\arctan\bigg(-\frac{1}{a}\sqrt{w^2-a^2}\bigg)+\pi
\end{equation}
and $\Delta\gamma_-:[a,\infty)\to[0,\infty)$, $\Delta\gamma_-(w)=-\Delta\gamma_+(w)$ for every $w\geq a$. As \eqref{eq:stability_of_complex_poly_gamma+_argument_a_pos} is a differentiable function its derivative is 
\begin{equation}\label{eq:gamma+_argument_a_pos_derivative}
\frac{d\Delta\gamma_+}{dw}(w)=\frac{w}{\sqrt{w^2-a^2}}\bigg(\tau_0-\frac{a}{w^2}\bigg).
\end{equation}
We have 
\begin{equation}\label{eq:gamma+_argument_a_pos_derivative_negative}
\frac{d\Delta\gamma_+}{dw}(w)<0\quad\hbox{if and only if}\quad w<w_m:=\sqrt{\frac{a}{\tau_0}},
\end{equation}
Taking into account the domain of \eqref{eq:stability_of_complex_poly_gamma+_argument_a_pos} i.e. $a\leq w$, we see that for $a\in(0,\frac{1}{\tau_0})$ function $\Delta\gamma_+$ is firstly decreasing, reaching a local minimum $\Delta\gamma_+(w_m)>\frac{\pi}{2}$, and then it is increasing to infinity; while for $a>\frac{1}{\tau_0}$ it is strictly increasing. These two cases are analysed separately.

\begin{figure}
    \centering
    \begin{minipage}{0.5\textwidth}
        \centering
        \includegraphics[width=0.9\textwidth]{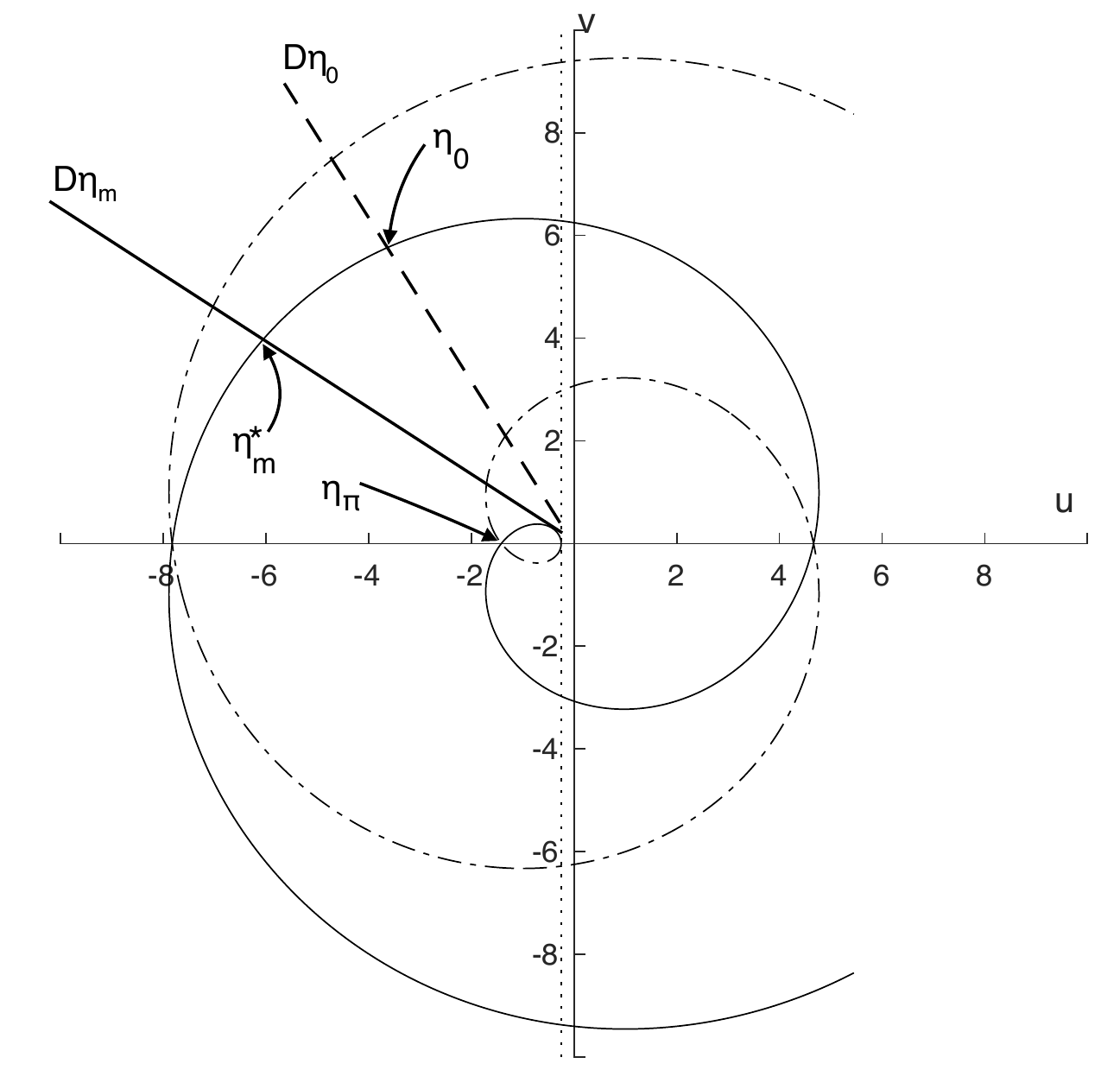} 
        \caption*{(a)}
    \end{minipage}\hfill
    \begin{minipage}{0.5\textwidth}
        \centering
        \includegraphics[width=0.9\textwidth]{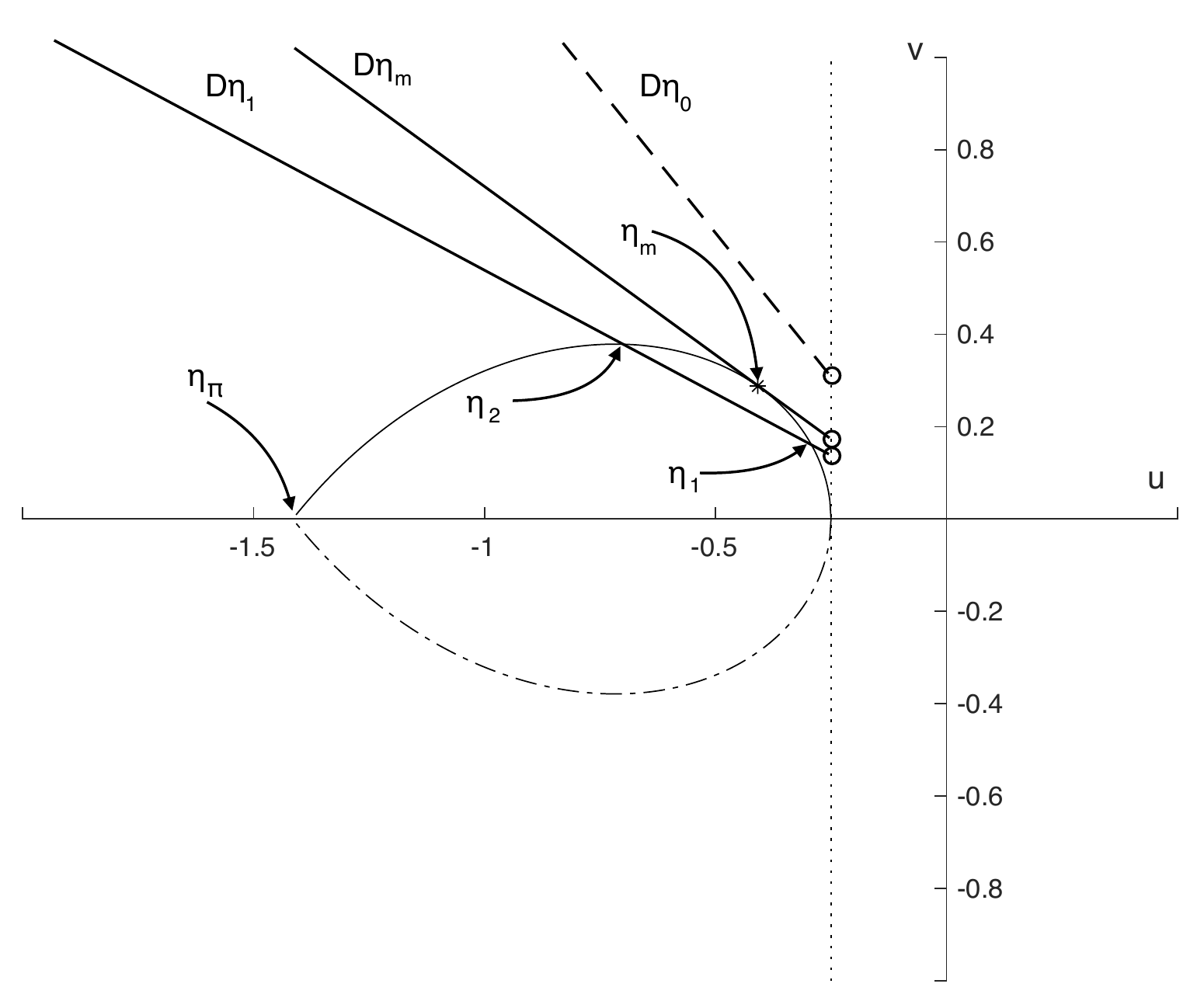} 
        \caption*{~\\~\\(b)}
    \end{minipage}
    \caption{(a): Curves $\Gamma_+$ (solid line) and $\Gamma_-$ (dash-dotted line) drawn for $\tau_0=1$ and $a=0.25$ with $w\in(|a|,10)$. The constraint related to $a$ and expressed by \eqref{eq:stability_of_complex_poly_necessary_condition} is marked with a dotted line. The first crossings of the real negative semi-axis by $\Gamma_+$ (and $\Gamma_-$) is at $\eta_\pi$. Auxiliary rays $D_{\eta_m}$ and $D_{\eta_0}$ are indicated in solid and dashed lines, respectively; (b): enlargement of the central part of (a) with $\gamma_+([|a|,|\eta_\pi|])\cup \gamma_-([|a|,|\eta_\pi|])$ with $D_{\eta_1}$ (solid line) and $D_{\eta_0}$ (dashed line), $|\eta_1|=w_1$, $|\eta_2|=w_2$. The $D_{\eta_0}$ ray is based on $\eta_0$ such that $\Arg\eta_0<\Delta\gamma_+(w_m)$; point $\eta_m=\gamma_+(w_m)$ is indicated with an arrow and a star $*$ symbol, $\Arg\eta_m=\Arg\eta_m^*$.}
    \label{fig:Gamma_curves_a_pos_small}
\end{figure}

\item[13.] Fix $0<a\leq\frac{1}{\tau_0}$. Similarly as in points 7 and 8 we focus initially on a part of $\Gamma_+$ given by $\gamma_+([a,\eta_\pi])$, as indicated in Fig.~\ref{fig:Gamma_curves_a_pos_small}. Take $\eta_1$ that fulfils \eqref{eq:stability_of_complex_poly_gamma+_path} and with $|\eta_1|=w_1<w_m$. For such $\eta_1$ we have
\[
\Delta\gamma_+(w_m)<\Arg\eta_1=\Delta\gamma_+(w_1)\leq\pi.
\]
Define a ray from the origin and passing through $\eta_1$ by
\[
D_{\eta_1}:=\{\eta=u+iv\in\CC:\abs{\eta}>\abs{a}\ \hbox{and}\ a+u<0\ \hbox{and}\ \Arg\eta=\Delta\gamma_+(w_1)\}
\]
and let $D_{\eta_1}^t:=\{t\geq0:t=|\eta|,\, \eta\in D_{\eta_1}\}$. To express $\tau$ as a function on this ray, i.e. $\tau:D_{\eta_1}^t\to (0,\infty)$ we now reformulate  \eqref{eq:stability_of_complex_poly_gamma+_argument_a_pos} to obtain
\begin{equation}\label{eq:tau_as_function_of_abs_eta_a_pos}
\tau(t):=\frac{\arctan\bigg(\frac{1}{a}\sqrt{t^2-a^2}\bigg)+\Delta\gamma_+(w_1)-\pi}{\sqrt{t^2-a^2}},
\end{equation}
where $\Delta\gamma_+(w_1)=\Arg\eta_1$.
Note also that as $w_1<w_m$ there exists $\eta_2$, with $|\eta_2|=w_2$, such that $\eta_2\in D_{\eta_1}\cap\gamma_+([a,\eta_\pi])$ and $w_m<w_2\leq\abs{\eta_\pi}$. 

The derivative of \eqref{eq:tau_as_function_of_abs_eta_a_pos} is again expressed by \eqref{eq:tau_as_function_derivative}, namely
\begin{equation*}
\frac{d\tau}{dt}(t)=\frac{t}{t^2-a^2}\bigg(\frac{a}{{t^2}}-\tau(t)\bigg),
\end{equation*} 
but, unlike in point 8, this derivative is in general not negative due to $a>0$. In fact, at the intersections $\{\eta_1,\eta_2\}=D_{\eta_1}\cap\gamma_+([a,\eta_\pi])$ we find
\begin{equation}
\begin{split}
\frac{d\tau}{dt}(w_1)=&\frac{w_1}{w_1^2-a^2}\left(\frac{a}{w_1^2}-\tau(w_1)\right)=\frac{w_1}{w_1^2-a^2}\left(\frac{a}{w_1^2}-\tau_0\right)\\
=&\frac{1}{\sqrt{w_1^2-a^2}}\left(-\frac{d\Delta\gamma_+}{dw}(w_1)\right)>0,
\end{split}
\end{equation}
where the last inequality comes from \eqref{eq:gamma+_argument_a_pos_derivative_negative}; similarly
\begin{equation}
\frac{d\tau}{dt}(w_2)=\frac{1}{\sqrt{w_2^2-a^2}}\left(-\frac{d\Delta\gamma_+}{dw}(w_2)\right)<0.
\end{equation}
We see that $\tau$ is an increasing function in a neighbourhood of $t_1=w_1$ and a decreasing one in a neighbourhood of $t_2=w_2$ i.e. at the boundaries of the $\Lambda_{\tau_0,a}$ region shown in Fig.~\ref{fig:Gamma_curves_a_pos_small}. If we show that $\tau$ has only one extreme value - a local maximum - inside  $\Lambda_{\tau_0,a}$, that is for some $t\in(w_1,w_2)$, then with the reasoning of point 8 we will show that for every $\eta$ inside $\Lambda_{\tau_0,a}$ region the roots of \eqref{eq:characteristic_equation_general} are in $\CC\setminus\CC_+$.

We are interested in the number of solutions of  $\frac{d\tau}{dt}(t)=0$, what is equivalent to the number of solutions of 
\begin{equation}\label{eq:condition_for_derivative_of_tau}
\frac{a}{{t^2}}=\frac{\arctan\bigg(\frac{1}{a}\sqrt{t^2-a^2}\bigg)+\beta}{\sqrt{t^2-a^2}},
\end{equation}
where $\beta=\Delta\gamma_+(w_1)-\pi$. Define $r:=\frac{1}{a}\sqrt{t^2-a^2}$. Then $r>0$ is a bijective image of $t>a$ and \eqref{eq:condition_for_derivative_of_tau} can be rearranged to 
\begin{equation}
\frac{r}{r^2+1}=\arctan(r)+\beta.
\end{equation}
As $\frac{\pi}{2}<\Delta\gamma_+(w_m)<\Arg\eta_1\leq\pi$ we have $\beta\in(-\frac{\pi}{2},0]$ and by Corollary~\ref{cor:rational_trig_eqaution_number_of_solutions} we infer that there is only one local extremum i.e. local maximum of $\tau$ for $t\in(w_1,w_2)$. Hence for every $\eta\in D_{\eta_1}$, $w_1\leq |\eta|\leq w_2$ we have $\tau(|\eta|)\geq\tau_0$ i.e.
\[
\Arg\eta_1\geq\sqrt{|\eta|^2-a^2}\tau_0-\arctan\left(\frac{1}{a}\sqrt{|\eta|^2-a^2}\right)+\pi.
\]
Thus by the definition of $D_{\eta_1}$ and symmetry about the real axis we obtain that for every $\eta$ with $|\eta|\leq|\eta_\pi|$ such that
\begin{equation}\label{eq:lambda_tau_a_pos_condition}
|\Arg\eta|\geq\sqrt{|\eta|^2-a^2}\tau_0-\arctan\left(\frac{1}{a}\sqrt{|\eta|^2-a^2}\right)+\pi
\end{equation}
the time $\tau$ for this $\eta$ to be such that the first root of \eqref{eq:complex_exponential_polynomial_zeros} reaches the imaginary axis is bigger than or equal to $\tau_0$. Argument similar to the one in point 8 shows that if $|\Arg\eta|\geq\Delta\gamma_+(w_m)$ then the only region we need to consider is the one given by \eqref{eq:lambda_tau_a_pos_condition}. Thus we distinguish a ray 
\[
D_{\eta_m}=\{\eta=u+iv\in\CC:|\eta|>|a|,\ a+u<0,\ \Arg\eta=\Delta\gamma_+(w_m)\}
\] 
together with a delay time function based on it, namely $\tau_m:D_{\eta_m}^t\to (0,\infty)$, 
\begin{equation}\label{eq:tau_as_function_of_abs_eta_a_pos_extreme_ray}
\tau_m(t)=\frac{\arctan\bigg(\frac{1}{a}\sqrt{t^2-a^2}\bigg)+\Delta\gamma_+(w_m)-\pi}{\sqrt{t^2-a^2}},
\end{equation}
where $\Delta\gamma_+(w_m)=\Arg\eta_m$, see Fig.~\ref{fig:Gamma_curves_a_pos_small}.  The above analysis shows that for $\tau_m$ we have $\tau_m(t)\leq\tau_0$ for every $t\in D_{\eta_m}^t$, where the equality holds only for $t=w_m$.

\item[14.] Take now, without loss of generality due to symmetry, $\eta\in\CC_-\cap\Pi_+$ such that $\re\eta<-a$ and $\frac{\pi}{2}<\Arg\eta<\Delta\gamma_+(w_m)=\Arg\eta_m$. We claim that for every such $\eta$ there is $\tau(|\eta|)<\tau_0$, where $\tau$ is defined on a ray containing $\eta$. Really, let us fix $\eta$ as above and assume otherwise i.e. $\tau(|\eta|)\geq\tau_0$. Then there exists $\eta_0$ that fulfils \eqref{eq:stability_of_complex_poly_step_1+}, $\Arg\eta_0=\Arg\eta$ and $\Delta\gamma_+(w_0)=\Arg\eta+2\pi$, where $w_0=|\eta_0|$ (see Fig.~\ref{fig:Gamma_curves_a_pos_small}). As $\eta\in D_{\eta_0}$ we have $\tau:D_{\eta_0}^t\to(0,\infty)$ defined as in \eqref{eq:tau_as_function_of_abs_eta_a_pos} but on the ray $D_{\eta_0}$, and such that for $t=|\eta|$ it takes the value
\begin{equation}\label{eq:tau_as_function_of_abs_eta_a_pos_contradiction}
\tau(t)=\frac{\arctan\bigg(\frac{1}{a}\sqrt{t^2-a^2}\bigg)+\Arg\eta_0+\pi}{\sqrt{t^2-a^2}},
\end{equation}
where we used a fact that $\Delta\gamma_+(w_0)=\Arg\eta_0+2\pi$. Note that for a fixed $t$ the above is a continuous function of $\Arg\eta_0\in (\frac{\pi}{2},\pi)$. Let us take a sequence $\{\eta_0^k\}_{k\in\NN}$ such that $\eta_0^k$ fulfils \eqref{eq:stability_of_complex_poly_step_1+}, $|\eta_0^k|<|\eta_0^{k+1}|$  for every $k\in\NN$ and $\eta_0^k\to\eta_m^*$ as $k\to\infty$, where $\Arg\eta_m^*=\Arg\eta_m$. Geometry of the problem shows that for every $k\in\NN$ we have
\[
\Arg\eta_0^k<\Arg\eta_0^{k+1}<\Arg\eta_m\quad\hbox{and}\quad D_{\eta_0^k}^t\subset  D_{\eta_0^{k}+1}^t\subset D_{\eta_m}^t.
\]
For the fixed $t$ from \eqref{eq:tau_as_function_of_abs_eta_a_pos_contradiction} consider a continuous, strictly increasing function $\tau_t:[\Arg\eta_0,\pi]\to(0,\infty)$, 
\[
\tau_t(\Arg\xi)=\frac{\arctan\bigg(\frac{1}{a}\sqrt{t^2-a^2}\bigg)+\Arg\xi+\pi}{\sqrt{t^2-a^2}}.
\] 
Our hypothesis now gives
\[
\tau_0\leq \tau(t)<\lim_{k\to\infty}\tau_t(\Arg\eta_0^k)=\tau_t(\Arg\eta_m^*)=\tau_m(t)\leq\tau_0,
\]
where we used strict monotonicity and continuity of $\tau_t$, continuity of $\gamma_+$, definition of $D_{\eta_m}$ and boundedness of $\tau_m$ given by \eqref{eq:tau_as_function_of_abs_eta_a_pos_extreme_ray}. The above contradiction proves our claim.

Thus with $0<a\leq\frac{1}{\tau_0}$ for the roots of \eqref{eq:complex_exponential_polynomial_zeros_tau0} to be in $\CC\setminus\CC_+$ the region given by \eqref{eq:lambda_tau_a_pos_condition} is the only allowable one for $\eta$. 

\item[15.] Fix $a>\frac{1}{\tau_0}$. By \eqref{eq:gamma+_argument_a_pos_derivative_negative} and a comment directly below it the continuous argument increment function $\Delta\gamma_+$ given by \eqref{eq:stability_of_complex_poly_gamma+_argument_a_pos} is now strictly increasing with range $\Delta\gamma_+([a,\infty))=[\pi,\infty)$. The minimal value of $\Delta\gamma_+(w)=\pi$ for $w=|a|$ and point 14 shows that if the roots of \eqref{eq:complex_exponential_polynomial_zeros_tau0} are in $\CC\setminus\CC_+$ then $\eta=-a$; there is no such $\eta$ that the roots of \eqref{eq:complex_exponential_polynomial_zeros_tau0} are in $\CC_-$. 

This finishes the necessity proof for \eqref{eq:complex_exponential_polynomial_zeros_tau0} and, by the same argument, for \eqref{eq:complex_exponential_polynomial_zeros}.

\item[16.] To be able to use previous notation and ease referencing we show sufficiency for  \eqref{eq:complex_exponential_polynomial_zeros_tau0}. Let $\tau_0>0$ be given and $a\leq\frac{1}{\tau_0}$.  The behaviour of the roots described in  points 1 and 2 does not change. Every $\eta\in\overbar{\Lambda_{\tau_0,a}}$, where $\Lambda_{\tau_0,a}$ is defined accordingly to $a$, is either inside $\DD_{|a|}$ or satisfies \eqref{eq:lambda_tau_a_condition}, \eqref{eq:lambda_tau_zero_condition} or \eqref{eq:lambda_tau_a_pos_condition}. Following backwards the reasoning in points 5--13 we reach the boundary condition \eqref{eq:lambda_tau_a_equality}, \eqref{eq:lambda_tau_zero_equality} or equality in  \eqref{eq:lambda_tau_a_pos_condition}, for which the roots of \eqref{eq:complex_exponential_polynomial_zeros_tau0} are on the imaginary axis, what happens exactly when $\eta$ is at the boudary of $\overbar{\Lambda_{\tau_0,a}}$.
\end{itemize}
\end{proof}
\begin{cor}\label{cor:stability_of_complex_polynomial}
Let a delay $\tau>0$, coefficients $\lambda,\gamma,\eta\in\CC$ be such that $\lambda=a+ib$ with $a\leq\frac{1}{\tau}$, $b\in\RR$, and let the corresponding $\Lambda_{\tau,a}\subset\CC$ be given by \eqref{eq:lambda_tau_a_neg}-\eqref{eq:lambda_tau_a_pos}. Then
\begin{itemize}
\item[(i)]  every solution of the equation $s-a -\eta\exp^{-s\tau}=0$ belongs to $\CC_-$ if and only if  $\eta\in\Lambda_{\tau,a}$;
\item[(ii)] every solution of  
\begin{equation}\label{eq:complex_polynomial_full_zeros}
s-\lambda-\gamma\exp^{-s\tau}=0
\end{equation}
and its version with conjugate coefficients
\begin{equation}\label{eq:complex_polynomial_full_conjugate_zeros}
s-\overbar{\lambda}-\overbar{\gamma}\exp^{-s\tau}=0
\end{equation}
belongs to $\CC_-$ if and only if $\gamma\exp^{-ib\tau}\in\Lambda_{\tau,a}$.
\end{itemize}
\end{cor}
\begin{proof}
Part $(i)$ follows from the continuity of \eqref{eq:tau_as_function_of_abs_eta} or \eqref{eq:tau_as_function_of_abs_eta_a_pos}. Part $(ii)$ follows from $(i)$ and Lemma~\ref{lem:complex_exponential_poly_equivalent_conditions} by defining $\eta=\gamma\exp^{-ib\tau}$ for the case of \eqref{eq:complex_polynomial_full_zeros}, while for the case of \eqref{eq:complex_polynomial_full_conjugate_zeros} by the real-axis symmetry of $\Lambda_{\tau,a}$ we have $\eta\in\Lambda_{\tau,a}$ if and only if $\overbar{\eta}=\overbar{\gamma}\exp^{ib\tau}\in\Lambda_{\tau,a}$.
\end{proof}
\section{Discussion}

Before going to examples we make some comments concerning previous work of other authors with respect to the proof of Theorem~\ref{thm:stability_of_complex_polynomial}. We also comment on  practicality of results obtained in this paper. 

Theorem~\ref{thm:stability_of_complex_polynomial} relies on subsets $\Lambda_{\tau,a}$ of the complex plane that are defined before the theorem itself. Their origin, however, becomes clear after going through points 6 -- 7 of the proof of Theorem~\ref{thm:stability_of_complex_polynomial}. The remainder of the proof is in fact an analysis of what happens inside those regions. It is worth to mention that inequalities in \eqref{eq:lambda_tau_a_neg}--\eqref{eq:lambda_tau_a_pos} can be obtained from the result in \cite{Nishiguchi_2016} after suitable simplifications.

As noted in the introduction an analysis of $\tau$ as a function of coefficients is present also in \cite{Matsunaga_2008}. The author obtains there an inequality similar to \eqref{eq:tau_as_function_of_abs_eta}, but does so in the context where $\gamma$ of \eqref{eq:differential-difference_eq_scalars} is a $2\times 2$ real matrix of a special form (and with $\lambda\in\RR$).

The necessary and sufficient condition for stability of first order scalar differential-difference equations with complex coefficients characterised by \eqref{eq:characteristic_equation_general} are given by Corollary~\ref{cor:stability_of_complex_polynomial}. The condition is based on mutual implicit relation between coefficients of the characteristic equation \eqref{eq:characteristic_equation_general} and a subset of $\CC$ given by \eqref{eq:lambda_tau_a_neg}-\eqref{eq:lambda_tau_a_pos}. As the latter is defined by non-linear inequalities there arises a question whether numerical approximation of the rightmost root isn't a more practical approach than finding numerically (i.e. approximately) the region given by \eqref{eq:lambda_tau_a_neg}-\eqref{eq:lambda_tau_a_pos}, especially given the abundance of literature of computational techniques to approximate characteristic roots.

Analysis of dependence between $\tau$ and the crossing of the imaginary axis by the first root in points 2--3 of the proof is well-known. In one of the early works \cite{Cooke_Grossman_1982} authors discuss \eqref{eq:characteristic_equation_general} with $\lambda,\gamma\in\RR$, in  \cite{Walton_Marshall_1987} the authors show a general approach for real polynamial case of \eqref{eq:differential-difference_equation_general_linearized} with multiple delays, what is also shown in \cite{Partington_2004}. More recently such analysis is also used in \cite{Matsunaga_2008}. A good exposition of such techniques is in \cite[Chapter 5.3.2]{Michiels_Niculescu_2014}.  Our calculations in point 2--3 are in fact based on \cite{Walton_Marshall_1987} and we decided to include all of its steps for the reader's convenience.

The answer to the above question depends, in the authors' opinion, on the purpose of approaching that problem. If the purpose is an analysis of a given differential-difference equation, considered as a delayed dynamical system, fulfilling assumptions of Corollary~\ref{cor:stability_of_complex_polynomial}, than a numerical check, up to a given accuracy, of at most one of inequalities \eqref{eq:lambda_tau_a_neg}-\eqref{eq:lambda_tau_a_pos} is usually a straightforward procedure.  

If, on the other hand, the purpose is a synthesis of a delayed dynamical system that has some \textit{a priori} specified properties, as may be the case of a controller design for such system, then a numerical search for the rightmost root may carry more relevant information.

\section{Examples}
With the above discussion in mind we present examples concerning only analysis of given differential-difference equations. These examples illustrate how the necessary and sufficient conditions of Theorem \ref{thm:stability_of_complex_polynomial} can be compared with and improve known literature results. 
Note initially that the stability condition discussed in \cite{Barwell_1975} and later proved in \cite{Cahlon_Schmidt_2002}, that is $-\re\lambda>|\gamma|$, follows immediately from $\abs{\eta}<\abs{a}$ (point 3 in the proof of Theorem~\ref{thm:stability_of_complex_polynomial}). 
Note also that as Corollary~\ref{cor:stability_of_complex_polynomial} concerns the placement of roots of the characteristic equation \eqref{eq:characteristic_equation_general}, it gives also a necessary and sufficient condition for stability of \eqref{eq:differential-difference_eq_scalars}. With that in mind we give the following examples.
\subsection{Example 1}
Consider a differential-difference equation
\begin{equation}\label{eq:example_1}
x'(t)=i20x(t)+\gamma x(t-0.1),
\end{equation}
where $\RR\ni\gamma>0$. Equation \eqref{eq:example_1} is a special case of \eqref{eq:differential-difference_eq_scalars}, for which necessary and sufficient conditions of stability were found in \cite{Cahlon_Schmidt_2002}. By Corollary~\ref{cor:stability_of_complex_polynomial} equation \eqref{eq:example_1} is stable if and only if $\gamma\exp^{-i2}\in\Lambda_{0.1,0}$, where $\Lambda_{0.1,0}$ is given by \eqref{eq:lambda_tau_a_zero}. We thus obtain that \eqref{eq:example_1} is stable if and only if $\gamma<20-5\pi$, what is equivalent to the condition given by \cite[Theorem 3.1]{Cahlon_Schmidt_2002}.

\subsection{Example 2}
Consider the differential-difference equation 
\begin{equation}\label{eq:example_2_diff}
x'(t)=\left(\frac{1}{4}+i\frac{\pi}{4}\right)x(t)-\left(\frac{1}{\sqrt{2}}+i\frac{1}{\sqrt{2}}\right)x(t-1),
\end{equation}
for which the corresponding characteristic equation takes the form
\begin{equation}\label{eq:example_2_char}
s-\left(\frac{1}{4}+i\frac{\pi}{4}\right)-\left(-\frac{1}{\sqrt{2}}-i\frac{1}{\sqrt{2}}\right)\exp^{-s}=0.
\end{equation}
By Corollary \ref{cor:stability_of_complex_polynomial} and \eqref{eq:lambda_tau_a_pos} (or, in fact, by investigating Fig.~\ref{fig:Lambda_tau_a_for_different_a}) we see that \eqref{eq:example_2_char} is stable.

\subsection{Example 3}
In \cite{Noonburg_1969} the author considers a semi-linear system version of \eqref{eq:differential-difference_equation_general} and - due to the approach method - states results only for a fixed delay $\tau=1$. The exemplary system analysed there is transformed to the form of \eqref{eq:differential-difference_equation_general_linearized} with $\tau=1$, $A=0$ i.e.
\begin{equation}\label{eq:example_3}
x'(t)=Bx(t-1),\quad B=\left(\begin{array}{cc}
        -1 & \frac{1}{8}\\
        -1 & -1
        \end{array}
\right).
\end{equation}
%
As $A=0$, and thus $\lambda=0$, we are interested only in eigenvalues of $B$, which are $-1\pm i\frac{1}{\sqrt{8}}$. The author concludes that the system is stable.

With conditions \eqref{eq:lambda_tau_a_zero} we can improve results for the exemplary system in \cite{Noonburg_1969} by finding a maximal delay $\tau$ for which such system 
remains stable. Let $\eta=-1+i\frac{1}{\sqrt{8}}$. Then $|\eta|=\frac{3\sqrt{2}}{4}$ and $\Arg\eta=\pi-\arctan\frac{1}{\sqrt{8}}$ and by \eqref{eq:lambda_tau_a_zero} we obtain that \eqref{eq:example_3} is stable if and only if 
\[
0<\tau<\frac{1}{|\eta|}\left(\Arg\eta-\frac{\pi}{2}\right)=\frac{2\sqrt{2}}{3}\left(\frac{\pi}{2}-\arctan\frac{1}{\sqrt{8}}\right).
\]
Note that we do not need to consider $\overbar{\eta}$ due to the symmetry of $\Lambda_{\tau,0}$ about the real axis.

\subsection{Example 4}
Previous examples relate current results to the ones known from the literature and thus demonstrate the technique. The following example shows how the current results can be used in the case of a retarded partial differential equation in an abstract formulation. 
 
Let the representation of our system be
\begin{equation}\label{eq:retarded_non-autonomous_system}
\left\{\begin{array}{ll}
        \dot{z}(t)=Az(t)+A_{1}z(t-\tau)+Bu(t)\\
        z(0)=x,        
       \end{array}
\right.
\end{equation}
where the state space $X$ is a Hilbert space, $A:D(A)\subset X\to X$ is a closed, densely defined diagonal generator of a $C_0$-semigroup $(T(t))_{t\geq0}$ on $X$, $A_1\in\calL(X)$ is also a diagonal operator and $0<\tau<\infty$ is a fixed delay. The input function is $u\in L^2(0,\infty;\CC)$ and $B$ is the control operator. We assume that $X$ posses a Riesz basis $(\phi_k)_{k\in\NN}$ consisting of eigenvectors of $A$, which has a corresponding sequence of eigenvalues $(\lambda_k)_{k\in\NN}$. 

A simplified form of \eqref{eq:retarded_non-autonomous_system} is analysed in \cite{Partington_Zawiski_2019} from the perspective of admissibility which, roughly speaking, asserts whether a solution $z$ of \eqref{eq:retarded_non-autonomous_system} follows a required type of trajectory. One of the key elements in the approach to admissibility analysis presented in \cite{Partington_Zawiski_2019} is to establish when a differential equation associated with the $k$-th component of \eqref{eq:retarded_non-autonomous_system}, namely 
\begin{equation}\label{eq:k-th_component_of_diagonal_non-autonomous_system}
\left\{\begin{array}{ll}
        \dot{z}_k(t)=\lambda_{k}z_{k}(t)+\gamma_{k}z_{k}(t-\tau)\\
        z_k(0)=x_k,        
       \end{array}
\right.
\end{equation}
is stable, where $\lambda_k\in\CC$ is an eigenvalue of $A$, $\gamma_k\in\CC$ is an eigenvalue of $A_1$ and $x_k\in\CC$ is an initial condition for the $k$-th component of $X$. Then, having stability conditions for every $k\in\NN$, one may proceed with analysis for the whole $X$. Based on Corollary~\ref{cor:stability_of_complex_polynomial} we immediately obtain a genuine approach method of obtaining these stability conditions, namely
\begin{prop}\label{prop:whole_system_infty_admissibility}
For a given delay $\tau\in(0,\infty)$ and sequences $(\lambda_k)_{k\in\NN}$ and $(\gamma_k)_{k\in\NN}$ consider a corresponding set of Cauchy problems of the form \eqref{eq:k-th_component_of_diagonal_non-autonomous_system}. For every $k\in\NN$ system \eqref{eq:k-th_component_of_diagonal_non-autonomous_system} is stable if and only if 
 \[
 \lambda_k=a_k+i\beta_k\in\left\{z\in\CC:\re(z)<\frac{1}{\tau}\right\}\quad\hbox{and}\quad  \gamma_k\exp^{-i\beta_k\tau}\in\Lambda_{\tau,a_k}\quad\forall k\in\NN,
 \]
with $\Lambda_{\tau,a_k}$ defined in \eqref{eq:lambda_tau_a_neg}-- \eqref{eq:lambda_tau_a_pos}. \qed
\end{prop}
Notice that Proposition~\ref{prop:whole_system_infty_admissibility} not only extends \cite[Proposition 3.5]{Partington_Zawiski_2019} by adding the necessary condition, but it also allows for analysis of unbounded $A$, as it includes e.g. the case when $a_k\to-\infty$ as $k\to\infty$. This is in fact exactly the case presented in \cite{Kapica_Partington_Zawiski_2022}.

\section{Acknowledgements}
%
The authors would like to thank Prof. Yuriy Tomilov for mentioning to them reference \cite{Noonburg_1969}. 
The research of Rafał Kapica was supported by the Faculty of Applied Mathematics AGH UST statutory tasks within subsidy of Ministry of Education and Science.
The work of Radosław Zawiski was performed when he was a visiting researcher at the Centre for Mathematical Sciences of the Lund University, hosted by Sandra Pott, and supported by the Polish National Agency for Academic Exchange (NAWA) within the Bekker programme under the agreement PPN/BEK/2020/1/00226/U/00001/A/00001.


\providecommand{\bysame}{\leavevmode\hbox to3em{\hrulefill}\thinspace}
\providecommand{\MR}{\relax\ifhmode\unskip\space\fi MR }
\providecommand{\MRhref}[2]{%
  \href{http://www.ams.org/mathscinet-getitem?mr=#1}{#2}
}
\providecommand{\href}[2]{#2}


\begin{thebibliography}{10}

\bibitem{Barwell_1975}
V.~K. {Barwell}, \emph{Special stability problems for functional differential
  equations}, BIT \textbf{15} (1975), 130--135.

%
\bibitem{Breda_2012}
D.~{Breda}, \emph{{On characteristic roots and stability charts of delay differential equations}},
  International Journal of Robust and Nonlinear Control, \textbf{22} (2012), 892--917.

\bibitem{Cahlon_Schmidt_2002}
B.~{Cahlon} and D.~{Schmidt}, \emph{On stability of a frst-order complex delay
  differential equation}, Nonlinear Analysis: Real World Applications \textbf{3}
  (2002), 413--429.
  
\bibitem{Cooke_Grossman_1982}
K. L.~{Cooke} and Z.~{Grossman}, \emph{Discrete delay, distributed delay and stability switches}, Journal of Mathematical Analysis and Applications \textbf{86} (1982), 592--627.
  (2002), 413--429.
  
  
\bibitem{Diekmann_et_al_1995}
O.~{Diekmann}, S.A.~{van Gils}, S.M.~{Verdyun Lunel} and H.-O.~{Walther}, \emph{Delay Equations
Functional-, Complex-, and Nonlinear Analysis}, Applied Mathematical Sciences, vol.~110, Springer-Verlag, New York, 1995.
  
\bibitem{Hayes_1950}
N.~D. {Hayes}, \emph{Roots of the transcendental equation associated with a
  certain difference-differential equation}, Journal of the London Mathematical
  Society \textbf{25} (1950), 226--232.
  
\bibitem{Kapica_Partington_Zawiski_2022}
R. {Kapica} and J.R. {Partington} and R. {Zawiski}, \emph{Admissibility of retarded diagonal systems with one dimensional input space}, arXiv 2207.00662 

\bibitem{Matsunaga_2008}
H. {Matsunaga}, \emph{Delay-dependent and delay-independent stability criteria for a delay differential system}, Proceedings of the American Mathematical Society \textbf{136} (2008), 4305--4312.
 
\bibitem{Motzkin_Taussky_1952}
T.~S. {Motzkin} and O.~{Taussky}, \emph{Pairs of matrices with property {L}},
  Transactions of the American Mathematical Society \textbf{73} (1952),
  108--114.

\bibitem{Nishiguchi_2016}
J.~{Nishiguchi}, \emph{On parameter dependence of exponential stability of
  equilibrium solutions in differential equations with a single constant
  delay}, {Discrete and Continuous Dynamical Systems} \textbf{36} (2016),
  5657--5679.

\bibitem{Noonburg_1969}
V.~W. {Noonburg}, \emph{Roots of a transcendental equation associated with a
  system of differential-difference equations}, SIAM Journal of Applied
  Mathematics \textbf{17} (1969), 198--205.

\bibitem{Partington_2004}
J.~R. {Partington}, \emph{{Linear Operators and Linear Systems: An Analytical
  Approach to Control Theory}}, London Mathematical Society Student Texts,
  vol.~60, Cambridge University Press, Cambridge, UK, 2004.

\bibitem{Partington_Zawiski_2019}
J.~R. {Partington} and R. {Zawiski}, \emph{{Admissibility of state delay diagonal systems with one-dimensional input space}}, Complex Analysis and Operator Theory \textbf{13} (2019), 
 2463-–2485.
       
\bibitem{Stepan_1989}
G.~{St{\'e}p{\'a}n}, \emph{Retarded dynamical systems: Stability and
  characteristic functions}, {Longman Scientific and Technical}, Harlow, 1989.

\bibitem{Tucsnak_Weiss}
M. {Tucsnak} and G. {Weiss}, \emph{{Observation and Control for Operator Semigroups}}, Birkh{\"a}user Verlag AG, Basel, 2009.

\bibitem{Walton_Marshall_1987}
K.~{Walton} and J. E.~{Marshall}, \emph{Direct method for {TDS} stability analysis}, IEE Proceedings {D} - control theory and applications \textbf{134} (1987), 101-107.

\bibitem{Wei_Zhang_2004}
J.~{Wei} and C.~{Zhang}, \emph{Stability analysis in a first-order complex
  differential equations with delay}, Nonlinear Analysis \textbf{59} (2004),
  657--671.
  
\bibitem{Michiels_Niculescu_2014}
W.~{Michiels} and S.-I.~{Niculescu}, \emph{{Stability, Control, and Computation for
  Time-Delay Systems}}, SIAM, Philadelphia, 2014.
  \end{thebibliography}
\end{document}